\documentclass[a4paper, 11pt]{article}
\usepackage[utf8]{inputenc}	
\usepackage[english]{babel} 

\usepackage{mathtools}
\usepackage{amsmath}
\usepackage{dsfont}
\usepackage{amsthm}
\usepackage{amssymb}
\usepackage{xcolor}
\usepackage{multirow}
\usepackage{bm}
\usepackage{lineno}
\usepackage[pdftex,pdfborder={0 0 0}, 
colorlinks=true, 
linkcolor=blue, 
citecolor=red, 
pagebackref=false, 
]{hyperref}

\usepackage{graphicx}
\usepackage{subcaption}

\usepackage{algorithm,algorithmic}

\theoremstyle{plain}
\newtheorem{theorem}{Theorem}
\newtheorem{proposition}[theorem]{Proposition}
\newtheorem{lemma}[theorem]{Lemma}
\newtheorem{corollary}[theorem]{Corollary}
\newtheorem{remark}[theorem]{Remark}

\newcommand{\R}{\mathbb{R}}

\newcommand{\gammamap}{\boldsymbol{\gamma}}
\newcommand{\Pmap}{\boldsymbol{P}}
\newcommand{\umap}{\boldsymbol{u}}
\newcommand{\vmap}{\boldsymbol{v}}
\newcommand{\mmap}{\boldsymbol{m}}
\newcommand{\wmap}{\boldsymbol{w}}
\newcommand{\fpmap}{\boldsymbol{M}}
\newcommand{\ballR}{\Xi_R}
\newcommand{\coupling}{\Xi}
\newcommand{\feedback}{\Theta}

\newcommand{\hypA}{(H1)}
\newcommand{\hypB}{(H2)}
\newcommand{\hypC}{(H3)}
\newcommand{\hypD}{(H4)}
\newcommand{\hypE}{(H5)}
\newcommand{\hypF}{(H6)}

\def\<{{\langle}}
\def\>{{\rangle}}
\def\dd{{\rm d}}

\def\1Lip{1\text{-Lip}}
\def\l({\left(}
\def\r){\right)}

\def\dd{\mathrm{d}}

\DeclareMathOperator*{\argmin}{arg\,min}

\usepackage{hyperref}
\usepackage{vmargin}
\setmarginsrb{3.5cm}{3cm}{3.5cm}{2cm}{0cm}{0cm}{0cm}{1.5cm}
\allowdisplaybreaks[3]

\begin{document}

\title{Generalized conditional gradient and learning in potential mean field games\footnote{This work was supported by a public grant as part of the
Investissement d'avenir project, reference ANR-11-LABX-0056-LMH,
LabEx LMH.} }
\author{
Pierre Lavigne\footnote{Institut Louis Bachelier. E-mail: \href{mailto:pierre.lavigne@institutlouisbachelier.org}{pierre.lavigne@institutlouisbachelier.org}.}
\and
Laurent Pfeiffer\footnote{Inria and Laboratoire des Signaux et Systèmes, CNRS (UMR 8506), CentraleSupélec, Université Paris-Saclay, 91190 Gif-sur-Yvette, France. E-mail: \href{mailto:laurent.pfeiffer@inria.fr}{laurent.pfeiffer@inria.fr}.}
}

\date{\today}

\maketitle

\begin{abstract}
We investigate the resolution of second-order, potential, and monotone mean field games with the generalized conditional gradient algorithm, an extension of the Frank-Wolfe algorithm. We show that the method is equivalent to the fictitious play method. We establish rates of convergence for the optimality gap, the exploitability, and the distances of the variables to the unique solution of the mean field game, for various choices of stepsizes. In particular, we show that linear convergence can be achieved when the stepsizes are computed by linesearch.
\end{abstract}



\section{Introduction}

\paragraph{Framework}
Mean field games (MFG), introduced by J.-M.~Lasry and P.-L.~Lions in \cite{LL07mf} and M.~Huang, R.~Malham\'e, and P.~Caines in \cite{HCMieeeAC06}, are a class of mathematical problems which allow to approximate differential games involving a very large number of agents. The general situation of interest is as follows: each agent aims at minimizing some cost function, depending on his own decision variables and some coupling terms, common to all agents. There are two fundamental assumptions in MFG theory: the coupling terms depend on the distribution of the agents and each agent has a negligible contribution to the coupling terms. Mean field games can typically be formulated as a coupled system of two equations, characterizing the decisions of the agents as functions of the coupling terms and vice versa.

In this article we consider a mean field game system of the following form 
\begin{equation} \label{eq:MFG-gcg} \tag{MFG}
	\begin{cases}
	\begin{array}{clc}
	\text{(i)} &
	\begin{cases}
	\, - \partial_t u - \Delta u + \bm{H}[  \nabla u + A^\star P] = \gamma, \\
	\,  u(x,T) =  g(x),
	\end{cases} & \begin{array}{r}
	(x,t) \in Q, \\
	x\in \mathbb{T}^d,
	\end{array}
	\\[1.5em]
	\text{(ii)} & v  = -  \bm{H}_p[  \nabla u +  A^\star P], & (x,t) \in Q,
	\\[1em]
	\text{(iii)} & \begin{cases}
	\, \partial_t m - \Delta m + \nabla \cdot (v m) = 0,  \\
	\, m(0,x) = m_0(x), 
	\end{cases} &  \begin{array}{r}
	(x,t) \in Q, \\
	x\in \mathbb{T}^d,
	\end{array} \\[1.5em]
	\text{(iv)} & \gamma(x,t) = f(x,t,m(t)), & (x,t) \in Q, \\[1em]
	\text{(v)} & P(t) = \bm{\phi}[A[v m]](t), & t \in [0,T].
	\end{array}
	\end{cases}
	\end{equation}
The unknown of the system is $(m,v,u,\gamma,P)$, with $m \colon Q \rightarrow \R$, $v \colon Q \rightarrow \R^d$, $u \colon Q \rightarrow \R$, $\gamma \colon Q \rightarrow \R$, and $P \colon [0,T] \rightarrow \R^k$. The data of the system (which will be clearly defined latter in Section \ref{sec:formulation}) are the initial distribution $m_0$ and the terminal cost $g$. The mapping $H$ denotes the Hamiltonian of the system,  $f$ and $\phi$ are a congestion cost and a price function.

In this model, the coupling terms are the variables $\gamma$ and $P$. Given $\gamma$ and $P$, the optimal control problem solved by a representative agent is described by \eqref{eq:cost_oc} and \eqref{mapping:u}; the optimal feedback $v$ is obtained by computing the corresponding value function $u$, solution the Hamilton-Jacobi-Bellman equation (\ref{eq:MFG-gcg},i) and then $v$ is obtained with using Equation (\ref{eq:MFG-gcg},ii).
Conversely, the coupling terms $\gamma$ and $P$ are deduced from $v$ by computing the distribution $m$ of the agents, solution to the Fokker-Planck equation (\ref{eq:MFG-gcg},iii). The first coupling term $\gamma$ is deduced from $m$ through (\ref{eq:MFG-gcg},iv) and the second coupling term $P$ is deduced from $m$ and $v$ through equation (\ref{eq:MFG-gcg},v).
Interactions through the density of players typically appear in epidemic or crowd motion models, while interactions through the controls $v$ typically appear in economics, finance or energy management models.

We assume in this work that the interaction cost and the price function derive from convex potentials. Our MFG system has then a potential structure, that is, it can be interpreted as the first-order necessary and sufficient optimality conditions for a convex mean field control problem, 
\begin{equation} \label{pb:control-w-FP} \tag{P}
\inf_{(m,w) \in \mathcal{R}} \, \mathcal{J}(m,w) \coloneqq \mathcal{J}_1(m,w) + \mathcal{J}_2(m,w),
\end{equation}
obtained using the classical Benamou-Brenier transformation.
Problem \eqref{pb:control-w-FP} is referred to as the potential problem. Roughly speaking (the precise definition is given in \eqref{label:defR}), the feasible set $\mathcal{R}$ is the set of pairs $(m,w)$ satisfying the Fokker-Planck equation $\partial_t m - \Delta m + \nabla \cdot w = 0$ and $m(0) = m_0$ and
\begin{equation} \label{eq:cost_decomposition}
\begin{array}{rl}
\mathcal{J}_1(m,w) = {} & {\displaystyle \int_{Q} \tilde{\bm{L}}[m,w](x,t) \, \dd x \, \dd t 
+ \int_{\mathbb{T}^d} g(x) m(x,T) \, \dd x, } \\[1.2em]
\mathcal{J}_2(m,w)= {} & {\displaystyle \int_0^T } \big( \bm{F}[m](t) + \bm{\Phi}[Aw](t) \big) \dd t.
\end{array}
\end{equation}

Potential MFGs have been widely investigated, we refer the reader to \cite{benamou2017variational,cardaliaguet2015second} 
for interactions through the density $m$ and to \cite{BHP-schauder,graber2018variational,graber2020weak} for price interactions.

Various methods from convex optimization have been been employed to solve  potential problems, see \cite{achdou2020mean} for a survey. A first approach consists in formulating the potential problem as a saddle-point problem and to solve it with primal-dual algorithms, see \cite{bonnans2021discrete,briceno2019implementation,briceno2018proximal}.
Another approach consists in applying the augmented Lagrangian algorithm to the dual problem of the potential problem, see \cite{Benamou2015,benamou2017variational,bonnans2021discrete}.
Other methods have been investigated such as the Sinkhorn algorithm \cite{benamou2019entropy}.

\paragraph{Generalized conditional gradient algorithm}
The aim of the paper is to show that the potential mean field game system \eqref{eq:MFG-gcg} can be efficiently solved with the \emph{generalized conditional gradient} method (GCG). This method is iterative: at iterate $k$, given a candidate $(\bar{m}_k,\bar{w}_k)$, one first solves the following partially linearized problem:
\begin{equation} \label{eq:the_first_lin_pb}
\inf_{(m,w) \in \mathcal{R}} \
\mathcal{J}_1(m,w)
+ D\mathcal{J}_2[\bar{m}_k,\bar{w}_k](m,w).
\end{equation}
The derivative $D\mathcal{J}_2[\bar{m}_k,\bar{w}_k]$ will be explicitly defined in the analysis.
Denoting by $(m_k,w_k)$ a solution to this problem, the next iterate is defined as
$(\bar{m}_{k+1},\bar{w}_{k+1}) = (1- \delta_k) (\bar{m}_{k},\bar{w}_{k}) + \delta_k (m_{k},w_{k})$ for some stepsize $\delta_k \in [0,1]$.

The GCG algorithm, first introduced in \cite{bredies2009generalized}, is an extension of the conditional gradient algorithm, also called Frank-Wolfe algorithm. The conditional gradient method allows to minimize a convex objective function on a convex and compact set. A classical choice of step size is given by $\delta_k = 2/(k+2)$ (see \cite{jaggi2013revisiting}) which yields the convergence of the objective function in $O(1/k)$. 

Most of the basic existing convergence results for the conditional gradient remain true for the GCG method, which also exhibits faster convergence rates in some cases, typically when the partially linearized cost function enjoys some coercivity property and when $\delta_k$ is obtained with a linesearch procedure. Improved rates of convergence have been recently obtained in \cite{kunisch2022fast}, in an infinite-dimensional setting.
The main interest of the generalized conditional gradient compared to the Frank-Wolfe algorithm is to use the regularity of the cost function to reach higher convergence rates. This algorithm is thus very suitable for potential mean field games since it is inherently assumed that the potential coupling terms are differentiable.

\paragraph{Learning in mean field games}

A fundamental issue in game theory is the formation of an equilibrium. It is often unrealistic to consider that the agents can perfectly anticipate the behavior of the others, in particular in the sophisticated situation underlying an MFG model. On contrast, it is more realistic to assume that the game is repeated many times and that the agents update their decisions according to a more or less complex procedure called \emph{learning} procedure. We consider in this article the \emph{fictitious play}, a learning procedure in which the agents play at each iteration of the game an optimal decision (also called \emph{best-responses}), corresponding to a predicted value of the coupling term (also called \emph{belief}), which is then updated.
In the context of MFGs, the fictitious play has been investigated in  \cite{cardaliaguet2017learning,elie2019approximate,HADIKHANLOO2019369,perrin2020fictitious}.
The convergence results for learning methods can be of various forms. In potential games, one can study the convergence of the potential cost along a sequence generated by the fictitious play algorithm.
In general, one can consider the \emph{exploitability} of the game at each iteration and try to show its convergence to zero.
For a given value of the coupling terms, the exploitability is the highest reduction of cost that a representative agent can achieve by changing his current decision to the best-response, assuming that the coupling terms remain the same. The convergence of the exploitability has been addressed in \cite{perrin2020fictitious} in the context of continuous-time learning and discrete mean field games, and a convergence rate is provided.

A key message of this article is that, in the context of second-order potential mean field games,  the fictitious play method can be interpreted as a GCG algorithm.
This is a consequence of the fact that the partially linearized problem \eqref{eq:the_first_lin_pb} is equivalent to a standard stochastic optimal control. As will be justified, the unique solution $(m_k,w_k)$ to \eqref{eq:the_first_lin_pb} is obtained by first computing the belief $(\gamma_k,P_k)$
\begin{equation*}
	\gamma_k(x,t) = f(x,t,\bar{m}_k(t)), \quad P_k(t) = \phi (t, A \bar{w}_k (t)),
\end{equation*}
next by solving the PDEs
\begin{equation*}
	\begin{cases}
		\begin{array}{lr}
			- \partial_t u_k - \Delta u_k + \bm{H}[  \nabla u_k + A^\star P_k] = \gamma_k,
	 	& \begin{array}{r} (x,t) \in Q, \end{array} \\
		 \partial_t m_k - \Delta m_k + \nabla \cdot w_k = 0, & \begin{array}{r} (x,t) \in Q, \end{array}
		\end{array}
	\end{cases}
\end{equation*}
using the same boundary conditions as in \eqref{eq:MFG-gcg}, and finally by computing the best-response $w_k  = -  \bm{H}_p[  \nabla u_k +  A^\star P_k] m_k$. If we further choose $\delta_k = 1/(k+1)$, we recover the fictitious play algorithm introduced in \cite{hadikhanloo2017learning}.

\paragraph{Exploitability and primal-dual gap}

The connection between the GCG algorithm and the ficitous play allows us to show that the notion of primal-dual gap \emph{primal-dual gap} and the notion of exploitability of the game are equivalent. 
This interpretation has already been highlighted in a very recent work \cite{geist2021concave}, for a class of potential mean field games with some discrete structure.
The connection between the Frank-Wolfe algorithm and fictitious play has also been been investigated in \cite{sorin2022continuous} for a general class of potential games.

\paragraph{Contributions}

The article \cite{cardaliaguet2017learning} is the most related to ours. It considers a second-order potential MFG, similar to our model but without price interaction.
It is proved that any cluster point (there exists at least one) of the sequences of value functions and probability distributions generated by the fictitious play is a solution to the MFG. In the case of a convex potential, the entire sequence converge.

The connection between the GCG algorithm and fictitious play, in the context of second-order MFGs, is the first contribution of our work. As we already mentioned, this connection was already established in \cite{geist2021concave}, in a different MFG setting. Taking advantage of this connection, we prove a general convergence result for the optimality gap (associated with the potential cost), when the stepsizes are predefined. These results easily lead to convergence rates, in particular, for $\delta_k= p/(k+p)$ (with $p> 0$), we prove a convergence rate of order $\mathcal{O}(k^{-p})$. This covers the case of the fictitious play (with $p=1$) and thus improves the convergence result of \cite{cardaliaguet2017learning} in the convex potential case.
Our most important contribution is the proof of the linear convergence of the optimality gap when the stepsizes are determined with classical linesearch rules, which is of major interest from a numerical perspective.
Let us note that our analysis is restricted to the case of non-degenerate second-order MFGs.
At a technical level, the proof of convergence utilizes in a crucial manner techniques from \cite{kunisch2022fast}.
Let us emphasize that for the methods from convex analysis which we have cited above (Chambolle-Pock, augmented Lagrangian, Sinkhorn), no rate of convergence has been established. To the best of our knowledge, only two references provide a linear rate of converge for MFGs: 
\cite{wang2021global}, which is restricted to linear-quadratic MFGs, and \cite{camilli2022rates}, under a restrictive smallness assumption for the coupling function.

\paragraph{Plan of the paper}

The paper is organized in two main parts. The first part, consisting of Sections \ref{sec:formulation} to \ref{sec:numerics}, present the GCG algorithm for MFGs. 
Section \ref{sec:formulation} is dedicated to the notations, the assumptions and the introduction of the main mappings used all along the article. We then describe the GCG algorithm in Section \ref{sec:gcg} and we explain its connection with fictitious play. We also state our main convergence results: Theorems \ref{theo:variables} and \ref{theo:main}. Section \ref{sec:numerics} present numerical illustrations of the convergence results for an academical problem.

The second part, consisting of Section \ref{sec:analysis}, deals with the analysis of the GCG algorithm. We establish the well-posedness of Algorithm \ref{algo:gcg} and prove Theorems \ref{theo:variables} and \ref{theo:main}.
Related technical details concerning the Fokker-Planck and the HJB equations are available in Appendix \ref{sec:mappings}.

\section{Notations, assumptions and mappings}

\label{sec:formulation}

\subsection{Notation}

Let $T>0$ denote the horizon of the game. We fix $d$ and $k$ in $\mathbb{N}^*$. We denote by $\mathbb{T}^d$ the $d$-dimensional torus and we set $Q =  \mathbb{T}^d \times  [0,T]$. For any subsets $O$ and $K$ of $\mathbb{R}^d$, we denote $\mathcal{C}(O;K)$ the set of continuous mappings on $O$ valued in $K$. In the article, when $K = \mathbb{R}$, we simply denote $\mathcal{C}(O)$ and we make use of this convention for any other functional spaces. 

\paragraph{H{\"o}lder spaces}

For any $\alpha \in (0,1)$, we denote by $\mathcal{C}^\alpha(Q)$ the set of H{\"o}lder continuous mappings on $Q$ of exponent $\alpha$. 
We denote by $\mathcal{C}^{\alpha,\alpha/2}(Q)$ the set of functions on $Q$ which are H{\"o}lder continuous of exponent $\alpha$ with respect to space and H{\"o}lder continuous of exponent $\alpha/2$ with respect to time. 
We denote by $\mathcal{C}^{2+\alpha,1+\alpha/2}(Q)$ the set of functions $u \in \mathcal{C}^\alpha(Q)$ with partial derivatives $\partial_t u$, $\partial_{x_i} u$, and $\partial_{x_i,x_j}^2 u$ in $\mathcal{C}^{\alpha,\alpha/2}(Q)$.
Finally, we denote by $\mathcal{C}^{2+\alpha}(\mathbb{T}^d)$ the set of $\alpha$-H\"older continuous functions, such that all partial derivatives up to the order two are $\alpha$-H\"older continuous.

\paragraph{Sobolev spaces and density space}
We denote by $W^{n,q}(\mathbb{T}^d)$ the Sobolev space of functions with weak partial derivatives in $L^q(\mathbb{T}^d)$, up to the order $n$. We set
\begin{align*} \nonumber
W^{2,1,q}(Q) = {} & W^{1,q}(Q) \cap L^q(0,T;W^{2,q}(\mathbb{T}^d)) \\ 
W^{1,0,\infty}(Q) = {} & L^\infty(0,T;W^{1,\infty}(\mathbb{T}^d)).
\end{align*}
From now on, we fix a real number $q$ such that $q > d + 2$.

\begin{lemma}
\label{lemma:max_reg_embedding}
There exists $\delta \in (0,1)$ and $C>0$ such that for all $u \in
W^{2,1,q}(Q)$,
\begin{equation*}
\| u \|_{\mathcal{C}^\delta(Q)} + \| \nabla u \|_{\mathcal{C}^\delta(Q;\R^d)}
\leq C \| u \|_{W^{2,1,q}(Q)}.
\end{equation*}
\end{lemma}

\begin{proof}
See \cite[Lemma II.3.3., page 80 and Corollary, page 342]{LSU}.
\end{proof}

We introduce the space $\feedback{}$, which will be used for the control variable of the system. It is defined by
\begin{equation} \label{eq:thetaBig}
\feedback{} = \Big\{ v \in \mathcal{C}(Q;\R^d) \,\big|\,
D_x v \in L^q(Q;\R^{d \times d})
\Big\},
\end{equation}
where $D_x v$ denotes the weak space derivative of $v$ with respect to $x$. We equip this space with the norm
\begin{equation*}
\| v \|_{\feedback}
= \| v \|_{L^\infty(Q;\R^d)}
+ \| D_x v \|_{L^q(Q;\R^{d\times d})}.
\end{equation*}
The coupling terms $(\gamma,P)$ of the MFG system of interest will be considered in the space $\coupling$, defined by
\begin{equation} \label{eq:ball}
\coupling=
\big( W^{1,0,\infty}(Q) \cap \mathcal{C}(Q) \big)
\times
\mathcal{C}(0,T;\R^k).
\end{equation}
In words, $(\gamma,P)$ lies in $\coupling$ if and only if $\gamma$ is continuous in both variables, Lipschitz continuous in $x$, uniformly in $t$, and $P$ is continuous.
Let $R>0$, we define the following subset of $\Xi$: \label{p:xi_r}
\begin{equation} \label{eq:ballR}
\ballR=
\Big\{
(\gamma,P) \in \Xi \, \Big| \,
\| \gamma \|_{W^{1,0,\infty}(Q)}
+ \| P \|_{L^\infty(0,T;\R^k)}
\leq R
\Big\}.
\end{equation}
We also define
\begin{equation} \nonumber
\mathcal{D}_1(\mathbb{T}^d) = \Big\{ m \in L^\infty(\mathbb{T}^d)\, \Big|\, m \geq 0 ,\; \int_{\mathbb{T}^d} m(x) \, \dd x = 1 \Big\}.
\end{equation}

\paragraph{Nemytskii operators}

Given two mappings $g \colon \mathcal{X} \times \mathcal{Y} \rightarrow \mathcal{Z}$ and $u \colon \mathcal{X} \rightarrow \mathcal{Y}$, we denote by $\bm{g}[u] \colon \mathcal{X} \rightarrow \mathcal{Z}$ the mapping  defined by
\begin{equation*}
\bm{g}[u](x)= g(x,u(x)), \quad \forall x \in \mathcal{X},
\end{equation*}
called Nemytskii operator. This notation will for instance be used for the Hamiltonian $H$: instead of writing $H(x,t,\nabla u(x,t))$, we write $\bm{H}[\nabla u](x,t)$.
Note that $\bm{H}_p$ will denote the Nemytskii operator associated with the partial derivative of $H$ with respect to $p$ (a similar notation will be used for the other partial derivatives).

\paragraph{Generic constants} All along the article, we make use of a generic constant $C>0$, depending only on the data of the problem. The value of $C$ may increase from an inequality to the next one. 
We will also make use of a generic constant $C(R)$ depending only on the data of the problem and some positive real number $R>0$.

\subsection{Assumptions}

We define two boundary conditions $m_0 \in \mathcal{D}_1(\mathbb{T}^d)$, $g \colon \mathbb{T}^d \rightarrow \mathbb{R}$, and four maps: a running cost $L$, an interaction cost $f$, a price function $\phi$ and an aggregation term $a$, 
	\begin{equation} \nonumber
	\begin{array}{rlrl}
	L\colon & \! \! \! Q \times \mathbb{R}^d \rightarrow \R, & \phi \colon & \! \! \! [0,T] \times \mathbb{R}^k \rightarrow \R^d, \\
	f \colon & \! \! \! Q \times \mathcal{D}_1(\mathbb{T}^d)\to \R, \qquad  & a \colon & \! \! \! Q \to \mathbb{R}^{k \times d}.
	\end{array}
	\end{equation}
	For any $(x,t,p) \in Q \times \mathbb{R}^d$, we define the Hamiltonian $H$ by
	\begin{equation} \nonumber
	H(x,t,p) = \sup_{v \in \mathbb{R}^d} \;  - \langle p, v \rangle - L(x,t,v).
	\end{equation}
 We also define the perspective function 
 $\tilde{L}\colon  Q \times \mathbb{R} \times \mathbb{R}^d \rightarrow \R $ of $L$, 
\begin{equation*} 
    \tilde{L}(x,t,m,w) = 
    \begin{cases}
    \begin{array}{ll}
    m L \big( x,t,
    \frac{w}{m}\big), & \text{if $m >0$}, \\
    0, & \text{if $m=0$ and $w=0,$} \\
    + \infty, & \text{otherwise.}
    \end{array}
    \end{cases}
\end{equation*}
Note that the function $\tilde{L}$ is convex and lower semi-continuous with respect to $(m,w)$.
	Next we define two linear operators \label{def:A}$A \colon w \in L^1(Q;\mathbb{R}^d) \mapsto A[w] \in L^1(0,T;\mathbb{R}^k)$ and $A^\star \colon P \in L^\infty(0,T;\mathbb{R}^k) \mapsto A^{*}[P] \in L^\infty(Q;\mathbb{R}^d)$ as follows:
	\begin{equation}  \nonumber
	A[w](t) =\int_{\mathbb{T}^d} a(x,t) w(x,t) \dd x, \quad  \quad A^\star[P](x,t) = a^\star(x,t) P(t), \quad \forall (x,t) \in Q.
	\end{equation}
	Note that the function $a$ will be assumed to be bounded.

We assume that there exist four constants $C_0>0$, $C_1>0$, $C_2>0$, and $\alpha_0 \in (0,1)$ such that the following holds true.

\begin{itemize}
\item[\hypA{}] \emph{Convexity of $L$.}
For any $(x,t) \in Q$, the function $L(x,t,\cdot)$ is strongly convex with modulus $1/C_0$.
\item[\hypB{}] \emph{Lipschitz continuity of $L$.}
For any $x$ and $y \in \mathbb{T}^d$, for any $t \in [0,T]$, and for any $v \in \R^d$, we have:
$|L(x,t,v)-L(y,t,v)| \leq C_0 |x-y| ( 1 + |v|^2)$.
\item[\hypC{}] \emph{Boundedness of $L$, $\phi$, and $f$.}
For any $(x,t) \in Q$, for any $v \in \R^d$, and for any $z \in \R^k$,
\begin{equation*}
L(x,t,v) \leq C_0 |v|^2 + C_0,
\quad
|\phi(t,z)| \leq C_0, \quad
\text{and} \quad
|f(x,t,m)| \leq C_0. 
\end{equation*}
\item[\hypD{}] \emph{Regularity assumptions.} The running cost $L$ is differentiable with respect to $v$ and $D_v L$ is differentiable with respect to $x$ and $v$. The mapping $a$ is differentiable with respect to $x$. The mapping $L$, $D_v L$, $D_{vx} L$, $D_{vv} L$, $\phi$, $a$, $D_x a$ are H{\"o}lder-continuous on any bounded set. The mappings $m_0$ and $g$ lie in $C^{2+\alpha_0}(\mathbb{T}^d)$.
\item[\hypE{}] \emph{Regularity of the coupling functions.} For all $(x_1,t_1)$ and $(x_2,t_2)$ in $Q$ and for all $m_1$ and $m_2$ in $\mathcal{D}_1(\mathbb{T}^d)$,
\begin{equation*}
|f(x_2,t_2,m_2)-f(x_1,t_1,m_1)| \leq
C_0 \big(
|x_2-x_1| + |t_2-t_1|^{\alpha_0} \big) + C_1 \| m_2 - m_1 \|_{L^2(\mathbb{T}^d)}.
\end{equation*}
For all $t \in [0,T]$ and for all $z_1$ and $z_2$ in $\R^k$, $|\phi(t,z_2)-\phi(t,z_1) | \leq C_2 |z_2 - z_1|$.
\item[\hypF{}] \emph{Potential structure.} The map $f$ is monotone with respect to its third variable, that is to say,
\begin{equation} \nonumber
\int_{\mathbb{T}^d}(f(x,t,m_2) -  f(x,t,m_1))(m_2(x)-m_1(x)) \, \dd x \geq 0,
\end{equation}
for any $m_1$ and $m_2 \in \mathcal{D}_1(\mathbb{T}^d)$ and $t \in [0,T]$.
We assume that $f$ has a primitive, that is, we assume the existence of a map $F \colon [0,T] \times \mathcal{D}_1(\mathbb{T}^d) \to \mathbb{R}$ such that
\begin{equation} \label{primitive:F}
F(t,m_2) - F(t,m_1) =\int_0^1 \int_{\mathbb{T}^d} f(x,t, s m_2 + (1-s)m_1)( m_2(x) - m_1(x)) \, \dd x \, \dd s.
\end{equation}
Moreover, $\phi$ has a convex potential $\Phi$, that is to say there exists a measurable mapping $\Phi : [0,T] \times \mathbb{R}^k \to \mathbb{R}$, convex and differentiable with respect to its second variable and such that $\phi(t,z) = \nabla_z \Phi(t,z)$ for any $(t,z) \in  [0,T] \times \mathbb{R}^k$.
\end{itemize}

\begin{remark}
\begin{enumerate}
\item An example of coupling term $f$ satisfying the above assumptions can be found in \cite[page 8]{BHP-schauder}.
\item The monotonicity assumption on $f$ implies that
\begin{equation} \nonumber
F(t,m_2) \geq F(t,m_1) + \int_{\mathbb{T}^d} f(x,t,m_1)(m_2(x)-m_1(x)) \, \dd x. 
\end{equation}
Since this inequality holds for any $m_1 \in \mathcal{D}_1(\mathbb{T}^d)$, $F$ is convex with respect to its second variable as the supremum of affine functions.
\item If the terminal condition $g$ depends on $m$, the results of the paper are still valid if $g$ satisfies the same kind of assumptions as $f$ (H3), (H5), and (H6) (boundedness, Lipschitz-continuity, monotonicity, existence of a potential function).
\item
Our assumptions are stronger than those of \cite{BHP-schauder}, since we require the boundedness of $\phi$ and the Lipschitz continuity of $f$ w.r.t.\@ $m$, for the $L^2(\mathbb{T}^d)$-norm.
\end{enumerate}
\end{remark}

\begin{theorem} \label{theo:existence} There exists $\alpha \in (0,1)$ such that the system \eqref{eq:MFG-gcg} has a unique solution $(\bar m,\bar v,\bar u,\bar \gamma,\bar P)$ in 
$\mathcal{C}^{2+\alpha,1+\alpha/2}(Q)
\times
\mathcal{C}^\alpha(Q;\R^d)
\times \mathcal{C}^{2+\alpha,1+\alpha/2}(Q)
\times
\mathcal{C}^\alpha(Q)
\times
\mathcal{C}^\alpha(0,T;\R^k)
$. Moreover, $Dv \in \mathcal{C}^{\alpha}(Q;\R^{d \times d})$.
\end{theorem}

\begin{proof}
Direct application of \cite[Theorem 1 and Proposition 2]{BHP-schauder}.
\end{proof}

\begin{lemma} \label{lemma:potential_struc}
The pair $(\bar{m},\bar{w})$ is the unique solution to the potential problem \eqref{pb:control-w-FP}.
\end{lemma}

\begin{proof}
Lemma \ref{lemma:potential_struc} is a direct consequence of Corollary \ref{coro:stability_mfg}, proved in Section \ref{sec:stability}.
\end{proof}

For the rest of the article, following Theorem \ref{theo:existence}, we denote by $(\bar m,\bar v,\bar u,\bar \gamma,\bar P)$ the unique solution to \eqref{eq:MFG-gcg} and we set $\bar{w}= \bar{m} \bar{v}$.

\subsection{Mappings} \label{subsec:mappings}

We introduce in this subsection different mappings, which will allow to express in a compact fashion the different mutual dependencies of the variables of the mean field game. The well-posedness of all these mappings will be justified in Appendix \ref{sec:mappings}.

\paragraph{Fokker-Planck mapping}
We first define the map $\fpmap \colon  v \in \feedback{} \mapsto \fpmap[v]  \in W^{2,1,q}(Q)$ which associates any vector field $v$ to the weak solution to the Fokker-Planck equation
\begin{equation} \label{eq:fokker-planck}
\begin{array}{rlr}
\partial_t m - \Delta m + \nabla \cdot (v m) \, = & \!\!\! 0 & (x,t) \in Q, \\
m(0,x) \, = & \!\!\! m_0(x) & x\in \mathbb{T}^d.
\end{array}
\end{equation}
Next we consider the set $\mathcal{R}$ defined by
\begin{equation} \label{label:defR}
\mathcal{R}=
\left\{
\begin{array}{l}
(m,w) \in W^{2,1,q}(Q) \times \feedback{} \, \big| \, \\
\qquad \partial_t m - \Delta m + \nabla \cdot w = 0,
\quad m(0,\cdot)= m_0 \\
\qquad \exists v \in L^\infty(Q;\R^d),\,
w= mv
\end{array}
\right\}.
\end{equation}

\begin{lemma} \label{lemma:convexityR}
The set $\mathcal{R}$ is convex. Moreover, given $v \in \feedback{}$, the pair $(m, mv)$ lies in $\mathcal{R}$, for $m = \fpmap[ v]$.
\end{lemma}

\begin{proof}
Let $(m_i,w_i) \in \mathcal{R}$, for $i=1,2$.
Let $v_i \in L^\infty(Q;\R^d)$ be such that $w_i= m_i v_i$.
Let $\theta \in (0,1)$ and let $(m_\theta,w_\theta)= (1-\theta)(m_1,w_1) + \theta(m_2,w_2)$. Clearly $\partial_t m_\theta - \Delta m_\theta + \nabla \cdot w_\theta = 0$ and $m_\theta(0,\cdot)= m_0$.
Let $C= \max_{i=1,2}\{\| v_i \|_{L^\infty(Q;\R^d)} \}$.
We define
\begin{equation*}
\varphi \colon (m,w) \in W^{2,1,q}(Q) \times \Theta \mapsto
\max_{(x,t) \in Q} \| w(x,t) \| - C m(x,t) \in \R.
\end{equation*}
We have $m_\theta \geq 0$. Define next $v_\theta$ as $v_\theta= w_\theta/m_\theta$ on $\{ m_\theta > 0 \}$, $v_\theta= 0$ on $\{ m_\theta=0 \}$. 
We have $\varphi(m_i,w_i) \leq 0$. By convexity of $\varphi$ we have $\varphi(m_\theta,w_\theta) \leq 0$, which implies that $\| v_{\theta} \|_{L^\infty(Q;\R^d)} \leq C$ and that for all $(x,t) \in Q$, if $m_\theta(x,t)=0$ then $w_\theta(x,t)=0$. Thus $w_{\theta}= m_{\theta}v_{\theta}$. The convexity of $\mathcal{R}$ follows.
The second part of the lemma is a consequence of Lemma \ref{lemma:max_reg_embedding} and Lemma \ref{lemma:fp}.
\end{proof}

\paragraph{HJB mapping}

Given $(\gamma,P) \in \coupling{}$, we define $\umap[\gamma,P]$ as the viscosity solution to the following HJB equation:
\begin{equation} \label{eq:hjb_alone}
\begin{array}{rlr}
- \partial_t u - \Delta u + \bm{H}[  \nabla u + A^\star P] \, = & \!\!\! \gamma & (x,t) \in Q, \\
 u(x,T) \, = & \!\!\! g(x) & x\in \mathbb{T}^d.
\end{array}
\end{equation}
Let us given an interpretation of $\umap$ as the value function of an optimal control problem.
Let $(B_s)_{s \in [0,T]}$ denote a $d$-dimensional Brownian motion. Let $\mathbb{F}$ denote the filtration generated by the Brownian motion $(B_s)_{s \in [0,T]}$.
We denote by $\mathbb{L}_{\mathbb{F}}^2(t,T)$ the set of progressively measurable stochastic processes $\nu$ defined on $[t,T]$ and valued in $\mathbb{R}^d$ such that $\mathbb{E} \big[ \int_t^T |\nu_s|^2 \dd s \big] < +\infty$.
Given $(\gamma,P) \in \coupling{}$, we consider the mapping $J[\gamma,P] \colon Q \times \mathbb{L}_{\mathbb{F}}^2(0,T) \to \mathbb{R}$, defined by
\begin{equation} \label{eq:cost_oc}
J[\gamma,P](x,t,\nu) =  \mathbb{E} \Big[\int_t^T \!\! \Big( L(X_s,\nu_s) + \langle A^\star [P](X_s,s) , \nu_s \rangle + \gamma(X_s,s) \Big) \dd s + g(X_T) \Big],
\end{equation}
where $(X_{s})_{s \in [t,T]}$ is the solution to
$\dd X_s = \nu_s \, \dd s + \sqrt{2} \, \dd B_s$, $X_t= x$.
Then, $\umap[\gamma,P]$ is the value function of the optimal control problem associated with $J[\gamma,P]$, that is to say,
for any $(x,t) \in Q$,
\begin{equation} \label{mapping:u}
\umap[\gamma,P](x,t)
=
\inf_{\nu \in \mathbb{L}_{\mathbb{F}}^2(0,T)} J[\gamma,P](x,t,\nu).
\end{equation}
This is a classical result from dynamic programming theory, see \cite{fleming2006controlled}.

\paragraph{The other mappings}

Next we introduce three mappings, $\vmap$, $\mmap$, and $\wmap$, defined on $\coupling$, and such that
\begin{equation*}
\mmap[\gamma,P] \in W^{2,1,q}(Q), \qquad
\vmap[\gamma,P] \in \feedback{}, \qquad \text{and} \qquad
\wmap[\gamma,P] \in \feedback{}.
\end{equation*}
For any pair $(\gamma,P) \in \coupling$, they are given by
\begin{align*}
\vmap[\gamma,P] ={} & -  \bm{H}_p[  \nabla \umap[\gamma,P] +  A^\star P], \\
\mmap[\gamma,P]={} & \fpmap \big[ \vmap[\gamma,P] \big], \\
\wmap[\gamma,P]={} & \mmap[\gamma,P] \vmap[\gamma,P].
\end{align*}
Finally, using Nemytskii operators, we define
\begin{equation*}
\begin{array}{l@{\hspace{2pt}}l@{\hspace{2pt}}}
\gammamap \colon &  W^{2,1,q}(Q) \rightarrow W^{1,0,\infty}(Q) \cap \mathcal{C}(Q), \\
& m \mapsto f[m], 
\end{array}
\qquad
\begin{array}{l@{\hspace{2pt}}l@{\hspace{2pt}}}
\Pmap \colon &  \feedback{} \rightarrow \mathcal{C}(0,T;\R^k) \\
&  w \mapsto \bm{\phi} \big[ A[w] \big].
\end{array}
\end{equation*}
In summary, the mappings $\umap$, $\vmap$, $\mmap$, and $\wmap$, derived from the equations (\ref{eq:MFG-gcg},i-iii), allow to express the behavior of the agents in function of the coupling terms $\gamma$ and $P$. Conversely, the mappings $\gammamap$ and $\Pmap$, derived from the equations (\ref{eq:MFG-gcg},iv-v), allow to express the coupling terms as a function of the behavior of the agents, described by the variables $m$ and $w$. 

\section{Generalized conditional gradient and fictitious play}
\label{sec:gcg}

In this section we present the GCG method, provide an interpretation as a learning method, and state our main results.

\subsection{Individual criterion and best reply}
\label{sec:gcg1}

We introduce now the PDE formulation of the stochastic optimal control problem solved by the representative agent.
Given a pair $(\gamma,P) \in \coupling$, we consider the criterion
\begin{equation*}
\mathcal{Z}[\gamma,P](m,w)
=
\mathcal{J}_1(m,w) + \Big( \int_{Q} \gamma(x,t) m(x,t) \, \dd x  \, \dd t
+ \int_0^T \langle A[w](t), P(t) \rangle \, \dd t \Big),
\end{equation*}
and the associated optimal control problem
\begin{equation} \label{pb:individual-control-problem-w}
\tag{$\mathcal{P}[\gamma,P]$}
\inf_{(m,w) \in \mathcal{R}} \ \mathcal{Z}[\gamma,P](m,w).
\end{equation}
Provided that it exists, any minimizer of the above problem can be interpreted as a best reply of the individual player, given the coupling term $(\gamma,P)$.
A key observation is that \eqref{pb:individual-control-problem-w} can be seen as a partial linearization of \eqref{pb:control-w-FP}. 

\begin{lemma}
For any $(\gamma,P) \in \coupling$, $(\mmap[\gamma,P], \wmap[\gamma,P])$ is the unique solution to \eqref{pb:individual-control-problem-w}.
\end{lemma}

\begin{proof}
This is a direct consequence of Proposition \ref{prop:stability_oc}, proved at page \pageref{proof:stability_oc}.
\end{proof}

Take a pair $(\hat{m},\hat{w}) \in \mathcal{R}$ and set $\hat{\gamma}= \gammamap{}[\hat{m}]$ and $\hat{P}= \Pmap[\hat{w}]$. The criterion $\mathcal{Z}[\hat{\gamma},\hat{P}](\cdot)$ can be seen as a partial linearization of $\mathcal{J}$: while the term $\mathcal{J}_1(m,w)$ is the same in both cost functions, $\int_Q \hat{\gamma} m \, \dd x \, \dd t$ is a linearization of $\int_0^T \bm{F}[m] \, \dd t$ around $\hat{m}$ and $\int_0^T \langle A[w], \hat{P} \rangle \, \dd t$ is a linearization of $\int_0^T \bm{\Phi}[Aw] \, \dd t$ around $\hat{w}$. 

\subsection{GCG algorithm and interpretation as a learning method}
\label{sec:gcg2}
Using the partial linearization of $\mathcal{J}$ introduced in the previous subsection, the GCG algorithm yields Algorithm \ref{algo:gcg}.

\begin{algorithm}[H]
\caption{Generalized conditional gradient} \label{algo:gcg}
\begin{algorithmic}
\STATE Choose $(\bar{m}_0,\bar{w}_0) \in \mathcal{R}$\;
\FOR{$0\leq k < N$} 
\STATE {
1. Set $\gamma_k= \gammamap[\bar{m}_k]$
and $P_k= \Pmap[\bar{w}_k]$. \\
2. Find the solution to $\mathcal{P}[\gamma_k,P_k]$, that is, define successively:
\begin{equation*}
u_k= \umap[\gamma_k,P_k], \quad
v_k= \vmap[\gamma_k,P_k], \quad
m_k= \mmap[\gamma_k,P_k], \quad
w_k= \wmap[\gamma_k,P_k].
\end{equation*}
3. Choose $\delta_k \in [0,1]$. \\
4. Update $(\bar{m}_{k+1},\bar{w}_{k+1})=(\bar{m}_{k}^{\delta_k},\bar{w}_{k}^{\delta_k})$, where
\begin{equation} \label{eq:def_conv_k}
(\bar{m}_{k}^{\delta},\bar{w}_{k}^\delta) = (1- \delta) (\bar{m}_{k},\bar{w}_{k}) + \delta (m_{k},w_{k}), \quad \forall \delta \in [0,1].
\end{equation}
}
\ENDFOR
\RETURN {$(\bar{m}_N, \bar{w}_N)$.}
\end{algorithmic}
\end{algorithm}

We can give an interpretation of Algorithm \ref{algo:gcg} as a learning method. At Step 1, $\bar{m}_k$ and $\bar{w}_k$ are can be seen as predictions of the equilibrium values $\bar{m}$ and $\bar{w}$. The agents use them to make a prediction of the coupling terms, $\gamma_k$ and $P_k$. In the second step, they find the corresponding best-response, by solving the HJB equation associated with their optimal control problem. Finally, in Steps 3 and 4, they update their prediction of $m$ and $w$. In particular, when $\delta_k= 1/(k+1)$, for any $k \in \mathbb{N}$, we are in the setting of the fictitious play, as investigated in \cite{cardaliaguet2017learning} (without price interaction).

We denote by $\varepsilon_k$ the optimality gap at iterate $k$, defined by
\begin{equation*}
\varepsilon_k= \mathcal{J}(\bar{m}_k,\bar{w}_k) - \inf_{ \mathcal{R}} \mathcal{J}
=
\mathcal{J}(\bar{m}_k,\bar{w}_k) - \mathcal{J}(\bar{m},\bar{w}).
\end{equation*}
We define the exploitability as follows
\begin{align*}
\sigma(m,w)
= \ &
\mathcal{Z}[\gamma,P](m,w)
-
\Big( \inf_{(m',w') \in \mathcal{R}}
\mathcal{Z}[\gamma,P](m',w') \Big) \\
= \ &
\mathcal{Z}[\gamma,P](m,w)
- \mathcal{Z}[\gamma,P](\mmap[\gamma,P],\wmap[\gamma,P]) \geq 0,
\end{align*}
where $\gamma= \gammamap[m]$ and $P= \Pmap[w]$. The exploitability is the largest decrease in cost that a representative agent can reach by playing its best response, assuming that all other agents play $(m,w)$. At equilibrium, since there is no profitable deviation, the exploitability is null. We denote by $\sigma_k$ the exploitability at iterate $k$, given by
\begin{align}
\sigma_k= {} & \mathcal{Z}[\gamma_k,P_k](\bar{m}_k,\bar{w}_k)
- \mathcal{Z}[\gamma_k,P_k](m_k,w_k). \label{eq:sigmak}
\end{align}
The convergence analysis will concern two kinds of stepsizes: predefined stepsizes, whose value only depends on $k$, and adaptive stepsize, whose value depend on the two pairs $(\bar{m}_k,\bar{w}_k)$ and $(m_k,w_k)$.
Following \cite{kunisch2022fast}, we will investigate three different rules for the determination of adaptive stepsizes.
\begin{itemize} \label{def:learning-seq}
\item \textbf{Optimal stepsizes:} Find $\delta_k$ such that
\begin{equation} \label{eq:stepsize_rule1}
\delta_k \in \argmin_{\delta \in [0,1]}   \ \mathcal{J}(\bar{m}_k^\delta,\bar{w}_k^\delta),
\end{equation}
where $(\bar{m}_k^\delta,\bar{w}_k^\delta)$ is defined as in \eqref{eq:def_conv_k}.

\item \textbf{Quasi-Armijo-Goldstein} condition: fix two parameters $c \in (0,1)$ and $\tau \in (0,1)$.
At iterate $k$, we say that $\delta \in [0,1]$ satisfies the Quasi-Armijo-Goldstein (QAG) condition if
\begin{equation*}
\mathcal{J}(\bar{m}_k^{\delta},\bar{w}_k^{\delta}) \leq \mathcal{J}(\bar{m},\bar{w}) - c \delta \sigma_k.
\end{equation*}
Then $\delta_k$ is defined by
\begin{equation} \label{eq:stepsize_rule2}
\delta_k
=
\tau^{i_k}, \quad
i_k
= \text{argmin} \,
\Big\{
j \in \mathbb{N} \,\big|\,
\tau^j \text{ satisfies the QAG condition}
\Big\}.
\end{equation}
\item \textbf{Exploitability-based} stepsizes: we take 
\begin{equation} \label{eq:stepsize_rule3}
\delta_k = \min \Bigg\{ 1,  \frac{\sigma_k}{2\big(C_1 D_k^{(1)} + C_2 D_k^{(2)} \big)} \Bigg\},
\end{equation}
where $C_1$ and $C_2$ are the constants of Assumption \hypE{} and where
\begin{align} \label{eq:D_k}
\begin{array}{rl}
D_k^{(1)}
= & \!\!\! {\displaystyle \int_0^T}
\| m_k(t,\cdot) - \bar{m}_k(t,\cdot) \|_{L^2(\mathbb{T}^d)}
\| m_k(t,\cdot) - \bar{m}_k(t,\cdot)) \|_{L^1(\mathbb{T}^d)}  \, \dd t
\\[1em]
D_k^{(2)}= & \!\!\!
{\displaystyle \int_0^T} \Big| \int_{\mathbb{T}^d} a(x,t) (w_2(x,t) - w_1(x,t)) \, \dd x \Big|^2 \, \dd t
\end{array}
\end{align}
\end{itemize}

\subsection{Main convergence results and discussion}\label{sec:gcg4}

Our first result concerns the convergence of the variables of the algorithm. Our second result provides several rates of convergence for the optimality gap, depending on the choice of  the learning sequence $(\delta_k)_{k \in \mathbb{N}}$.
The two theorems are proved in Subsections \ref{sec:convergence} and \ref{subsec:proof_main}.

\begin{theorem} \label{theo:variables}
There exists $C>0$ such that for all $k \in \mathbb{N}$,
\begin{align*}
\| \bar{m}_k - \bar{m} \|_{L^\infty(0,T;L^2(\mathbb{T}^d))}
+
\| \bar{w}_k - \bar{w} \|_{L^2(Q;\R^d)}
\leq {} & C \sqrt{\varepsilon_k}, \\
\| \gamma_k - \bar{\gamma} \|_{L^\infty(Q)} +
\| P_k - \bar{P} \|_{L^\infty(0,T;\R^k)} \leq
{} & C \sqrt{\varepsilon_k}, \\
\| u_k - \bar{u} \|_{L^\infty(Q)} \leq {} & C
\sqrt{\varepsilon_k}, \\
\| m_k - \bar{m}_k \|_{L^\infty(0,T;L^2(\mathbb{T}^d))}
+
\| w_k - \bar{w}_k \|_{L^2(Q;\R^d)} \leq {} & C \sqrt{\sigma_k}. 
\end{align*}
\end{theorem}

\begin{theorem} \label{theo:main}
\begin{enumerate}
\item At each iteration, assume that $(\delta_k)_{k \in \mathbb{N}}$ satisfy one the three adaptive rules described above: either \eqref{eq:stepsize_rule1}, \eqref{eq:stepsize_rule2} (for values of $\tau$ and $c$ independent of $k$), or \eqref{eq:stepsize_rule3}. Then, there exist two constants $C > 0$ and $\lambda> 0$ such that
\begin{equation*}
\varepsilon_k \leq C \lambda^k, \quad
\forall k \in \mathbb{N}.
\end{equation*}
Moreover, the number $i_k$ of iterations for the QAG condition is bounded by some constant independent of $k$.
\item There exists a constant $C>0$ such that whatever the choice of stepsizes $\delta_k \in [0,1]$, it holds that
\begin{equation} \label{eq:conv_general}
\varepsilon_{k+1}
\leq \frac{\varepsilon_0 \exp \big( C \sum_{j=0}^k \delta_j^2  \big) }{\exp \big( \sum_{j=0}^k \delta_j \big)}, \quad \forall k \in \mathbb{N}.
\end{equation}
\end{enumerate}
\end{theorem}

Let us first discuss some aspects related to adaptive stepsizes.
While the method for computing $\delta_k$ in the case of the QAG condition or the exploitability-based formula is explicit, it is impossible to find $\delta_k$ that exactly satisfies \eqref{eq:stepsize_rule1}. 
We suggest to compute an approximation of the optimal stepsize with the golden-section method, which we briefly describe. Denote by $\varphi \coloneqq (\sqrt{5} + 1)/2$ the golden number and choose a tolerance $\kappa \in (0,1)$.
At each step $k \in \{0, \ldots, N \}$, the learning rate $\delta_k$ is computed as follows:
Set $(a,d) = (0,1)$ and $(b,c) = (d - (d-a)/\varphi,a + (d-a)/\varphi)$.  While $a-d > \kappa$, find
\begin{equation} \label{eq:stepsize_rule1bis}
\bar{\delta} \in \argmin_{\delta \in \{a,b,c,d\}} \,   \mathcal{J}(\bar{m}_k^\delta,\bar{w}_k^\delta).
\end{equation}
Then set $d = b$ (resp.\@ $d = c$, $a = b$ or $a = c$) if $\bar{\delta} = a$ (resp.\@ $\bar{\delta}= b$, $\bar{\delta}= c$ or $\bar{\delta}= d$).
When $a-d \leq \kappa$, stop and set $\delta_k^{\text{GS}}= \bar{\delta}$.
Finally, denote $\delta_k^{\text{QAG}}$ the stepsize determined by \eqref{eq:stepsize_rule2}. As a substitute for \eqref{eq:stepsize_rule1}, one can choose $\delta_k \in \argmin_{\delta \in \{ \delta_k^{\text{GS}},  \delta_k^{\text{QAG}} \} } \, \mathcal{J}(\bar{m}_k^{\delta}, \bar{w}_k^{\delta}).$
In view of the proof of Theorem \ref{theo:main}, it is clear that linear convergence is also achieved for this choice of stepsize.

\begin{remark}
The practical computation of $\delta_k$ requires a bounded number (with respect to $k$) of evaluations  of the cost functional $\mathcal{J}$ in each of the three considered cases, \eqref{eq:stepsize_rule2} (since $i_k$ is bounded), \eqref{eq:stepsize_rule3} (the formula is explicit), \eqref{eq:stepsize_rule1bis} (the number of evaluations can be bounded by some constant depending on $\kappa$ and $\varphi$). Therefore linear convergence is not only achieved with respect to $k$ but also with respect to the number of evaluations of $\mathcal{J}$.
\end{remark}

We discuss now the convergence of $(\varepsilon_k)_{k \in \mathbb{N}}$ for predefined stepsizes. The following lemma covers the case of the fictitious play learning rate ($\delta_k= \frac{1}{k+1}$) and the Frank-Wolfe stepsize ($\delta_k= \frac{2}{k+2}$).

\begin{lemma} \label{lemma:p_conv}
Let $p>0$. Assume that $\delta_k= \frac{p}{k+p}$, for any $k \in \mathbb{N}$. Then, for all  $k \in \mathbb{N}$,
\begin{equation*}
\varepsilon_k
\leq \frac{ \varepsilon_0 \, p^p \exp(2Cp)}{(k+p+1)^p}.
\end{equation*}
\end{lemma}

\begin{proof}
We have the following inequalities:
\begin{align*}
& \sum_{j=0}^k \, \frac{p}{j+p}
\geq \int_0^{k+1} \frac{p}{s+p} \, \dd s
= p \ln \Big( \frac{k+p+1}{p} \Big) \\
& \sum_{j=0}^k \, \left( \frac{p}{j+p} \right)^2
\leq p^2 \Big( \frac{1}{p} + \int_0^\infty \frac{1}{(s+p)^2} \, \dd s \Big) = 2p.
\end{align*}
Using inequality \eqref{eq:conv_general}, we deduce that $\varepsilon_k \leq \varepsilon_0
\exp \big( 2Cp  - p \ln (k+p) + p \ln(p) \big)$, from which the announced result follows.
\end{proof}

\begin{remark} \label{remark:cv_open}
Lemma \ref{lemma:p_conv} shows that rates of convergence of order $\mathcal{O}(k^{-p})$ can be achieved. Yet the constant behind the asymptotic rate of convergence, of order $p^p \exp(2Cp)$, increases quickly  with $p$, which mitigates the interest of taking a too large value of $p$ in practice. This is confirmed in the numerical tests of Section \ref{sec:numerics}.   
\end{remark}

\begin{lemma} \label{lemma:cv_gal}
Let the sequence of stepsizes $(\delta_k)_{k \in \mathbb{N}}$ converge to zero. Then, for any $r \in (0,1)$, there exist two constants $C>0$ and $k_0 \in \mathbb{N}$, both depending on the sequence of stepsizes and $r$, such that
\begin{equation*}
\varepsilon_{k+1}
\leq
\frac{C \varepsilon_0}{\exp\big[ (1-r) \big( \sum_{j=0}^k \delta_j \big) \big]},
\quad \forall k \geq k_0.
\end{equation*}
In particular, if $\sum_{k=0}^\infty \delta_k = \infty$, then $\varepsilon_k \underset{k \to \infty}{\longrightarrow} 0$.
\end{lemma}

\begin{proof}
Let $r \in (0,1)$. Let $k_0 \in \mathbb{N}$ be such that $\delta_k \leq \frac{r}{C}$, for any $k \geq k_0$ (and where $C$ is as in \eqref{eq:conv_general} holds true). Then for any $k \geq k_0$, we have
\begin{equation*}
C \sum_{j=0}^k \delta_j^2
\leq C \sum_{j=0}^{k_0-1} \delta_j^2 + \underbrace{C\delta_{k_0}}_{\leq r} \sum_{j=k_0}^{k} \delta_j.
\end{equation*}
It follows that
$C \sum_{j=0}^k \delta_j^2
- \sum_{j=0}^k \delta_j
\leq \Big( C \sum_{j=0}^{k_0-1} \delta_j^2 + r \sum_{j=0}^{k_0-1} \delta_j \Big)
- (1-r) \sum_{j=0}^{k} \delta_j$.
Plugging this inequality into \eqref{eq:conv_general},
we obtain the announced result.
\end{proof}

Lemma \ref{lemma:cv_gal} shows that high convergence rates can be achieved if the sequence $(\delta_k)_{k \in \mathbb{N}}$ decreases slowly, as we already noticed in Lemma \ref{lemma:p_conv}.
However, if the convergence to 0 is slow, the values of $k_0$ and $C$ may be very large, as was revealed in the proof. This fact is verified for the sequence $\delta_k=(1+k)^{-\alpha}$.

\section{Numerical illustration}
\label{sec:numerics}

As a numerical illustration, we solve the mean field game system \eqref{eq:MFG-gcg} with $f = 0$. In this situation, the agents only interacts through the law of their controls. The associated potential problem has the following form
\begin{equation*} 
\inf_{(m,w) \in \mathcal{R}} \ \int_{Q} \tilde{\bm{L}}[m,w] \, \dd x \, \dd t + \int_0^T \bm{\Phi}[Aw] \, \dd t
+ \int_{\mathbb{T}^d} g m(T) \, \dd x.
\end{equation*}

\paragraph{Data and numerical scheme}
We take $d=2$, $k=2$ and $T=1$ so that $Q = \mathbb{T}^2 \times [0,1]$. The initial measure $m_0$ is normally distributed on the torus (it is the product of two independent von Mises distributions centered at $1/4$ and is shown in Figure \ref{fig:m_0}). The terminal condition $g(x) = \sum_{i = 1}^2 \cos(2\pi x_i)$ for any $x = (x_1,x_2) \in \mathbb{T}^2$ is shown in Figure \ref{fig:g}. 
\begin{figure}[h!]
\centering
\begin{minipage}[c]{0.9\linewidth}
\centering
\begin{subfigure}{0.45\textwidth}
\includegraphics[width=0.99\linewidth]{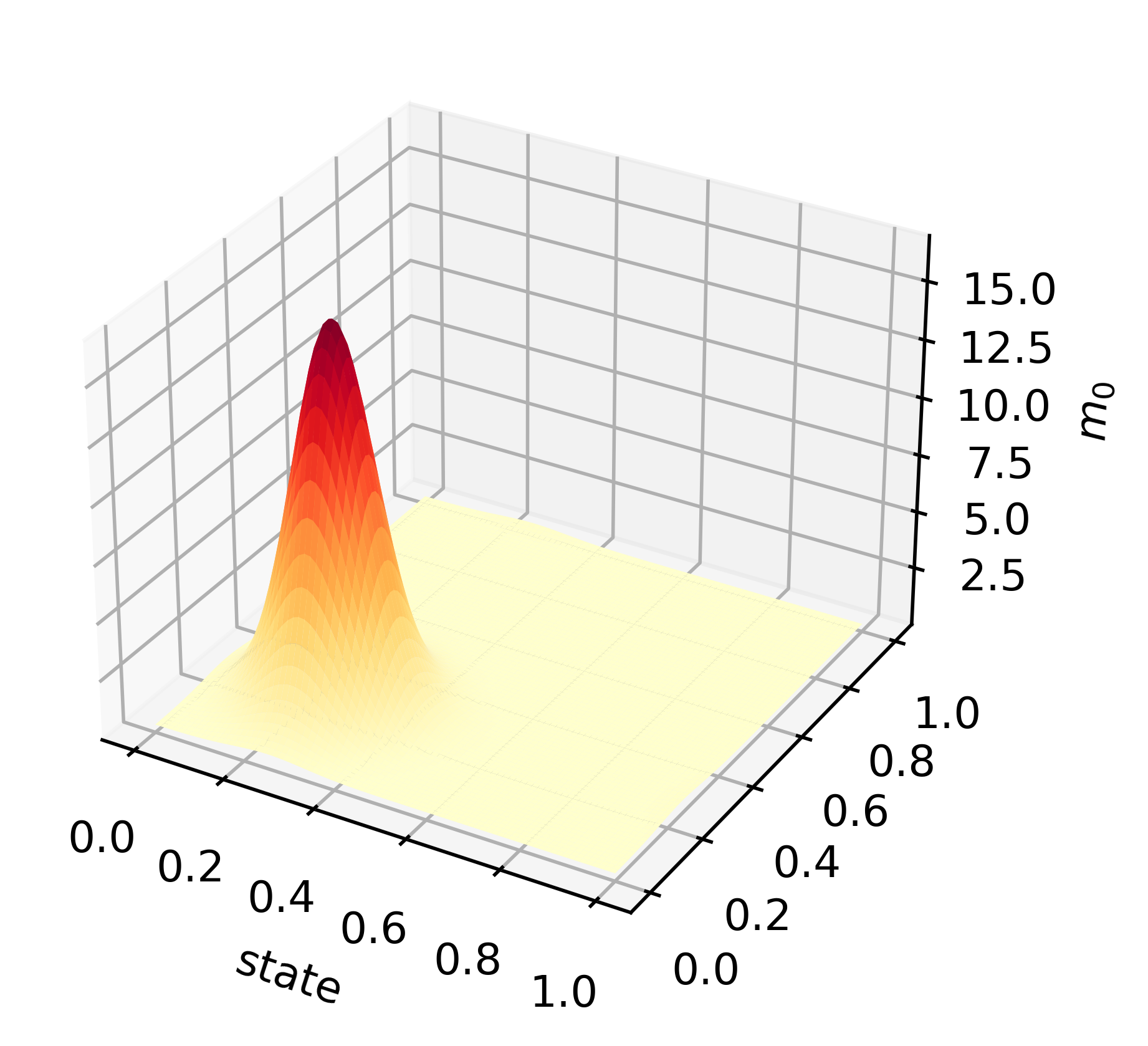} 
\caption{Initial measure $m_0$.}
\label{fig:m_0}
\end{subfigure}
\begin{subfigure}{0.45\textwidth}
\includegraphics[width=0.99\linewidth]{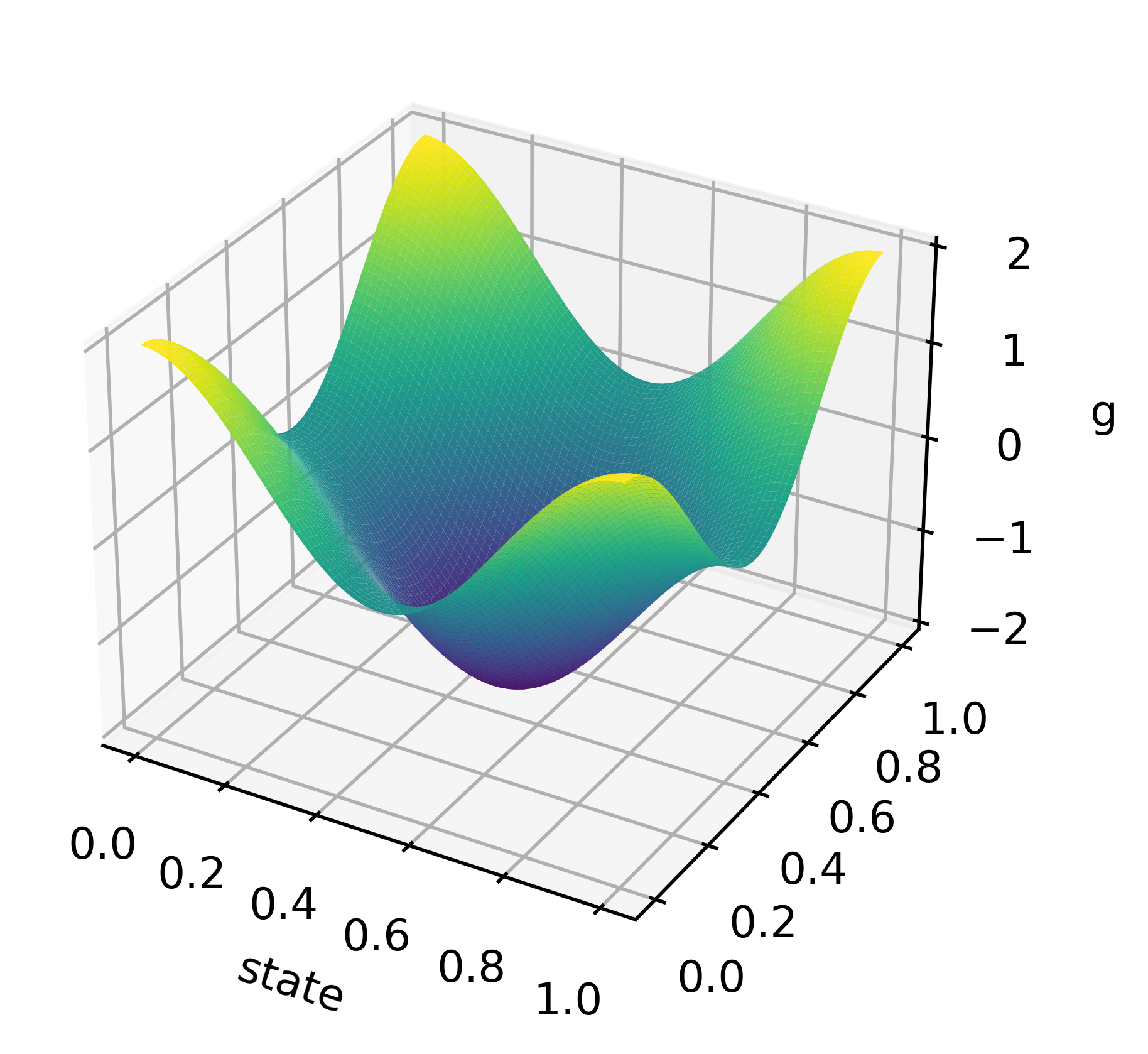}
\caption{Terminal condition $g$.}
\label{fig:g}
\end{subfigure}
\end{minipage}
\end{figure}

We define $L(x,t,v)= \frac{1}{2} |v|^2$, $a(x,t)= \text{Id}$ and $\phi(t,z)= 10z$. Obviously $\phi$ derives from the potential $\Phi(t,z)= 5(z_1^2 + z_2^2)$. In other words, we have a two-dimensional price variable, and the two price relations write, for $i=1,2$ as follows: $P_i= 10 \int_{\mathbb{T}^d} v_i(x,t) m_i(x,t) \, \dd x \, \dd t$. For any control $\nu_t \in \mathbb{L}^2_{\mathbb{F}}(0,1)$, the cost function of a representative agent writes:
\begin{equation*}
\mathbb{E} \Big[\int_0^T  \frac{1}{2} |\nu_t |^2 + 10\,  \langle P(t) , \nu_t \rangle \, \dd t + g(X^{\nu}_T) \Big].
\end{equation*}

In this numerical experiment we consider a volatility equal to $0.1$ for the controlled stochastic state equation satisfied by $X^{\nu}$.
Algorithm \ref{algo:gcg} requires to compute the mappings $\bm{u}$ (i.e. a solution to the Hamilton-Jacobi-Bellman equation) and $\bm{M}$ (i.e. a solution to the Fokker-Planck equation) at each step. The resolution is done via an explicit finite difference scheme. 
In the following, we discretize $Q$ with a uniform grid containing $10^2$ points in space and $42$ points in time.

\paragraph{Interpretation and numerical solution}

 \begin{figure}
 \centering
 \begin{minipage}[c]{0.9\linewidth}
 \begin{subfigure}{0.99\textwidth}
 \centering
    \includegraphics[width=0.8\linewidth]{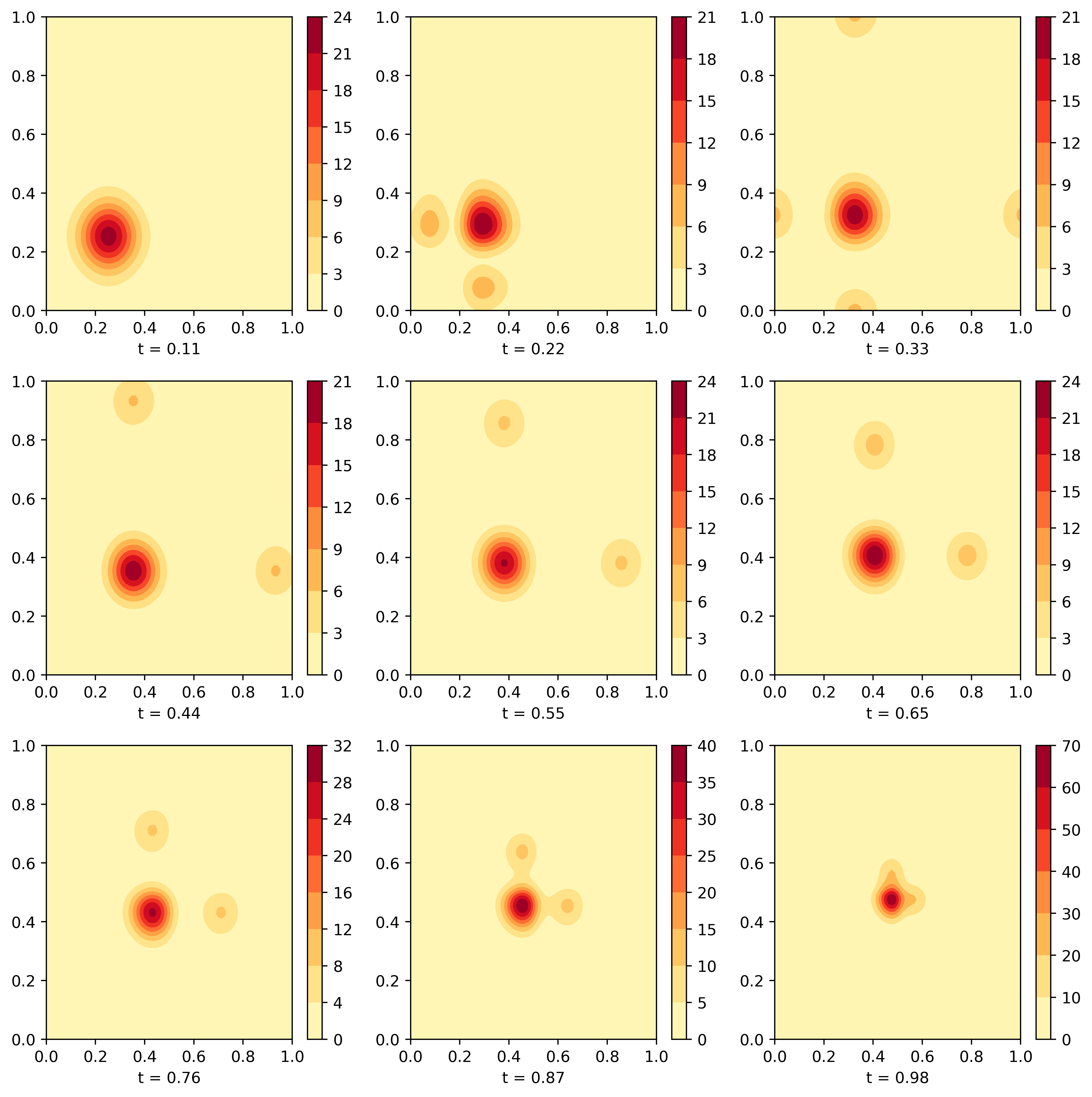}
    \caption{Density of players $\bar{m}$ evaluated at several time steps.}
    \label{grap:m-bar}
 \end{subfigure}
\\[1em] 
 \begin{subfigure}{0.99\textwidth}
 \centering
 \includegraphics[width=0.8\linewidth]{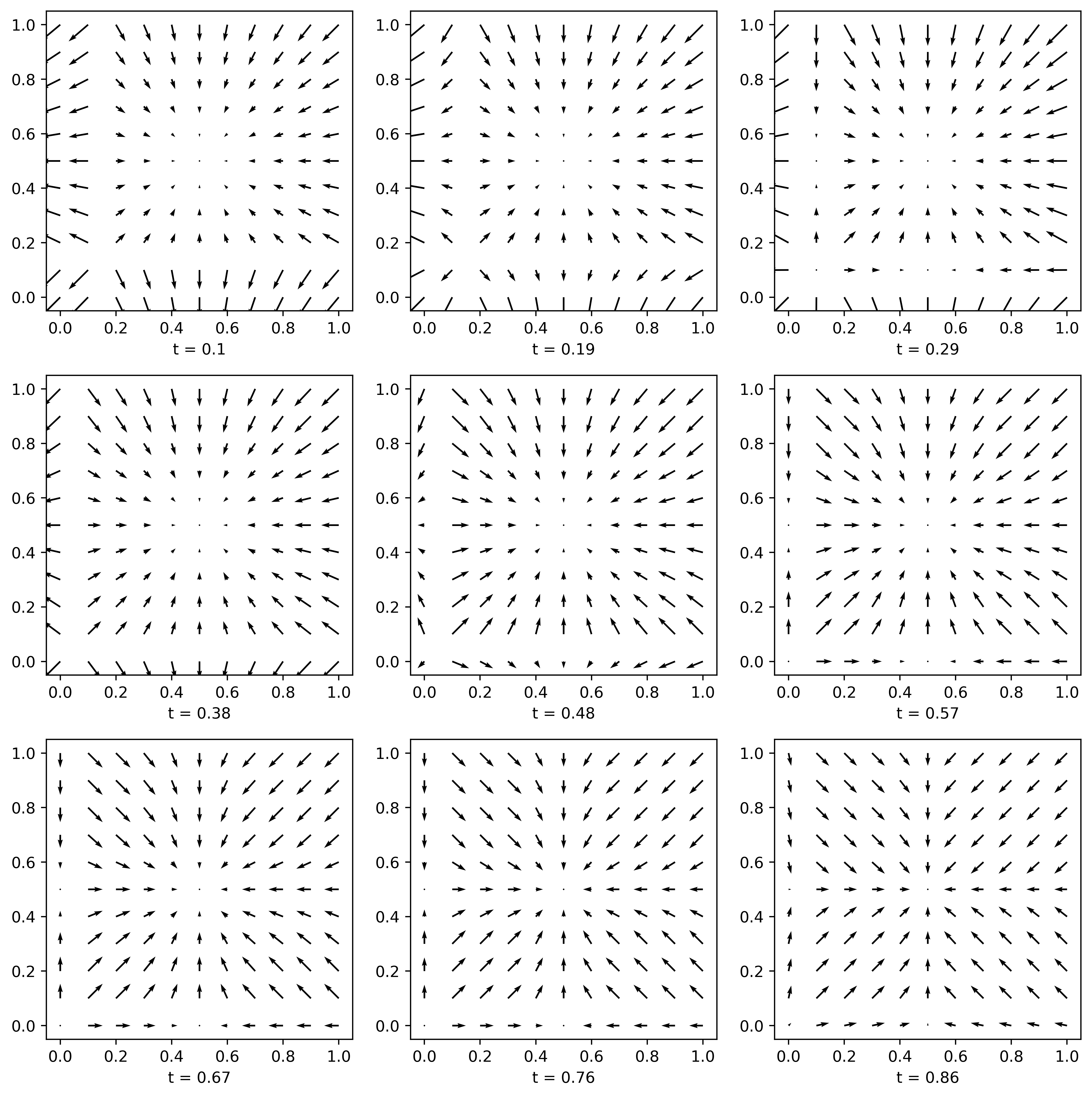}
\caption{Vector field $\bar{v}$ evaluated at several time steps.}
\label{grap:w-bar}
 \end{subfigure}
 \end{minipage}
 \caption{Density and velocity vector field evaluated at several points in time.}
\end{figure}

When $P=0$, the optimal trajectories of the agents look like slightly perturbed straight lines (with constant speed), ending in a close neighborhood of the point $(0.5,0.5)$ (it would be a straight line in the deterministic case without diffusion coefficient). In view of the initial condition, located in the ``bottom right corner" of the square $[0,1]^2$, the agents use in this case positive controls $v_1$ and $v_2$. However, when $P_1$ is positive, some agents may try to reach the point $(0.5,0.5)$ using control with a first coordinate that is negative. Graphically speaking, these agent would cross the left vertical axis ($x_1=0$) and ``jump" to the right vertical axis $(x_1=1)$. This strategy is particularly interesting for agents an initial condition $x$ such that $x_1$ is positive and close to zero. Of course the same reasoning is valid for $P_2$ positive: some agents would cross the horizontal axis.

The two equilibrium prices are positive, leading to four different kinds of optimal trajectories: those which do not cross any axis, those crossing only the vertical axis, those crossing only the horizontal axis, and those crossing both axes, as can be seen from the graphs of the equilibrium vector field $\bar{v}$ on Figure \ref{grap:w-bar}. Thus the initial distribution is split into four groups as shown by Figure \ref{grap:m-bar}. The group of agents crossing both axes is actually of very small mass, thus not visible on the graph.

\paragraph{Convergence and execution time}

\begin{figure}
 \centering
 \begin{minipage}[c]{0.9\linewidth}
 \begin{subfigure}{0.45\textwidth}
 \includegraphics[width=0.99\linewidth]{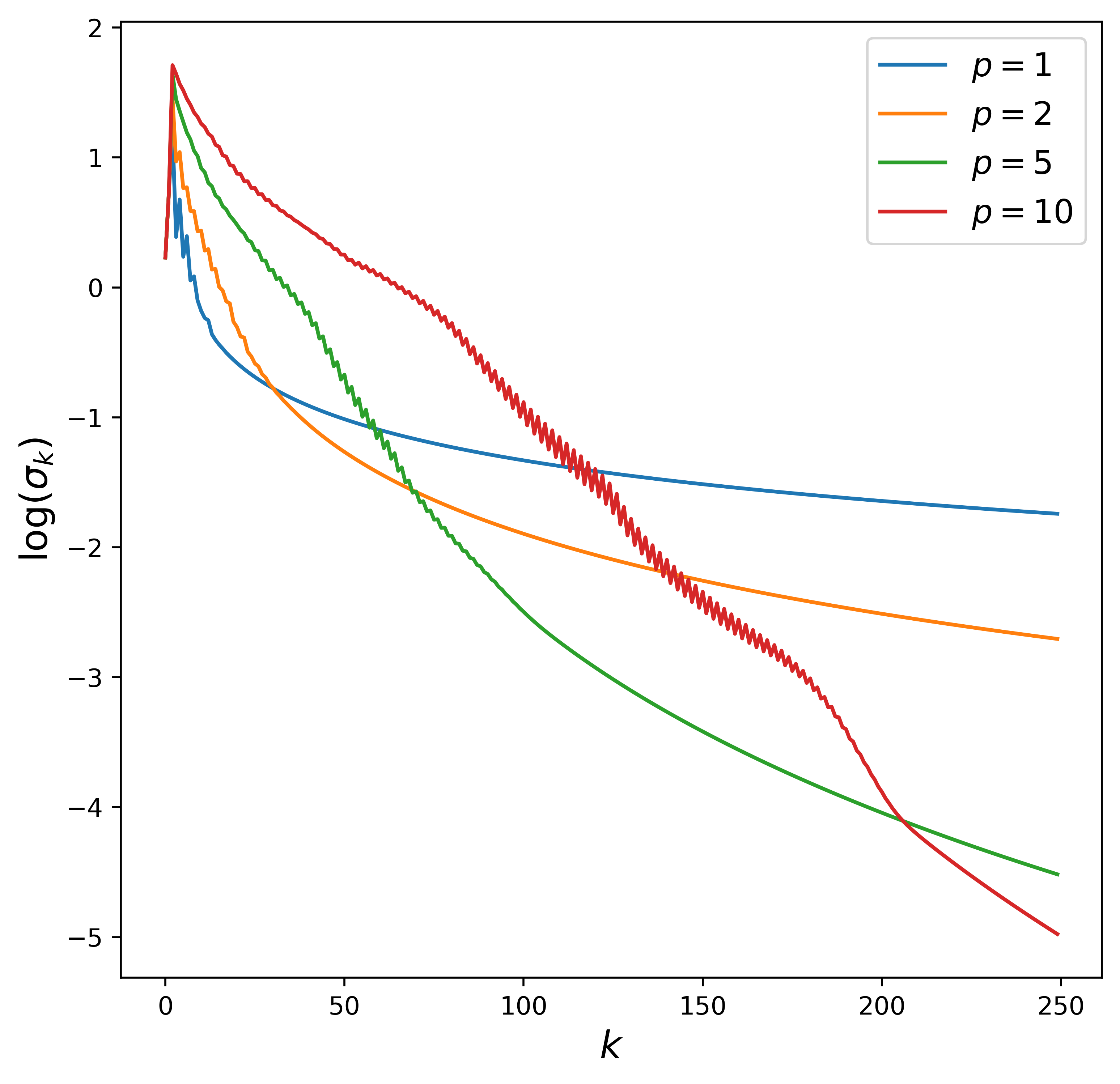}
 \caption{Logarithm of the exploitability for $\delta_k = p/(k+p)$.}
 \label{fig:p-exp}
 \end{subfigure}
 \hspace*{\fill}
 \begin{subfigure}{0.45\textwidth}
 \includegraphics[width=0.99\linewidth]{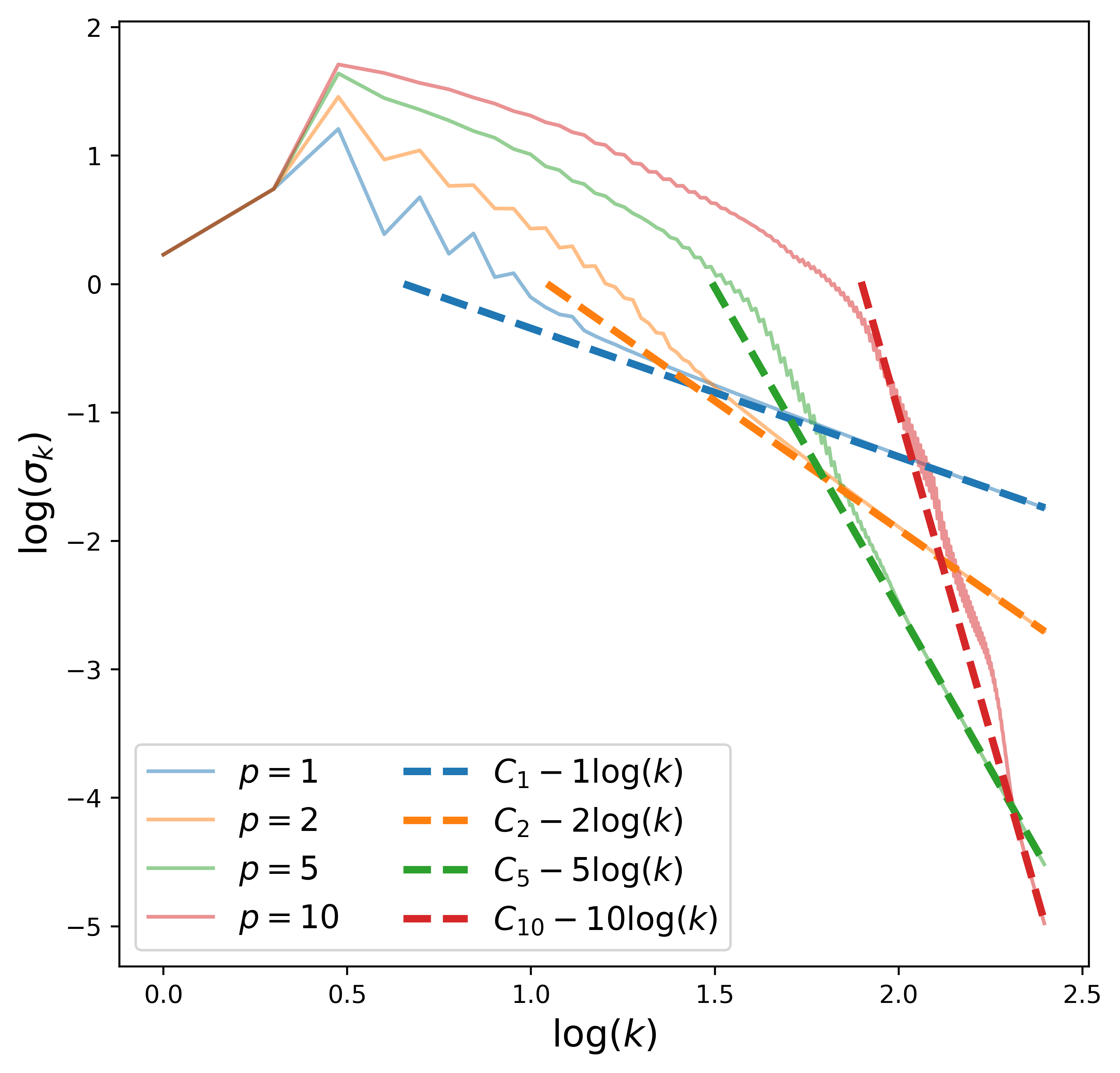}
 \caption{Log-log scale of the exploitablity for $\delta_k = p/(k+p)$.}
 \label{fig:log-log-p-exp}
 \end{subfigure}
 \\[1em]
 \begin{subfigure}{0.45\textwidth}
 \includegraphics[width=0.99\linewidth]{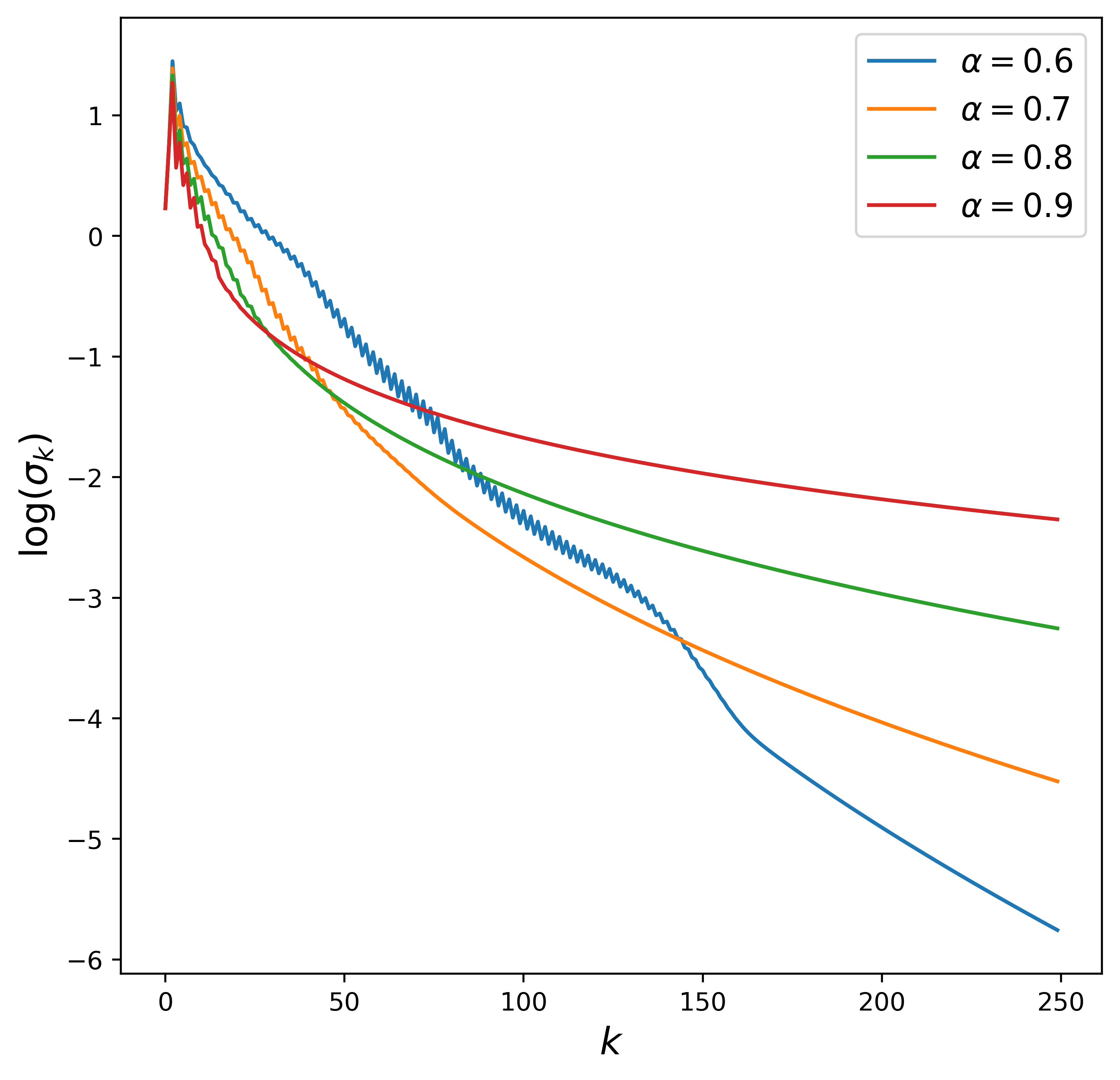}
 \caption{Logarithm of the exploitability for the prescribed stepsizes $\delta_k = (k+1)^{-\alpha}$.}
 \label{fig:alpha-exp}
 \end{subfigure}
 \hspace*{\fill}
 \\[1em]
 \begin{subfigure}{0.45\textwidth}
 \includegraphics[width=0.99\linewidth]{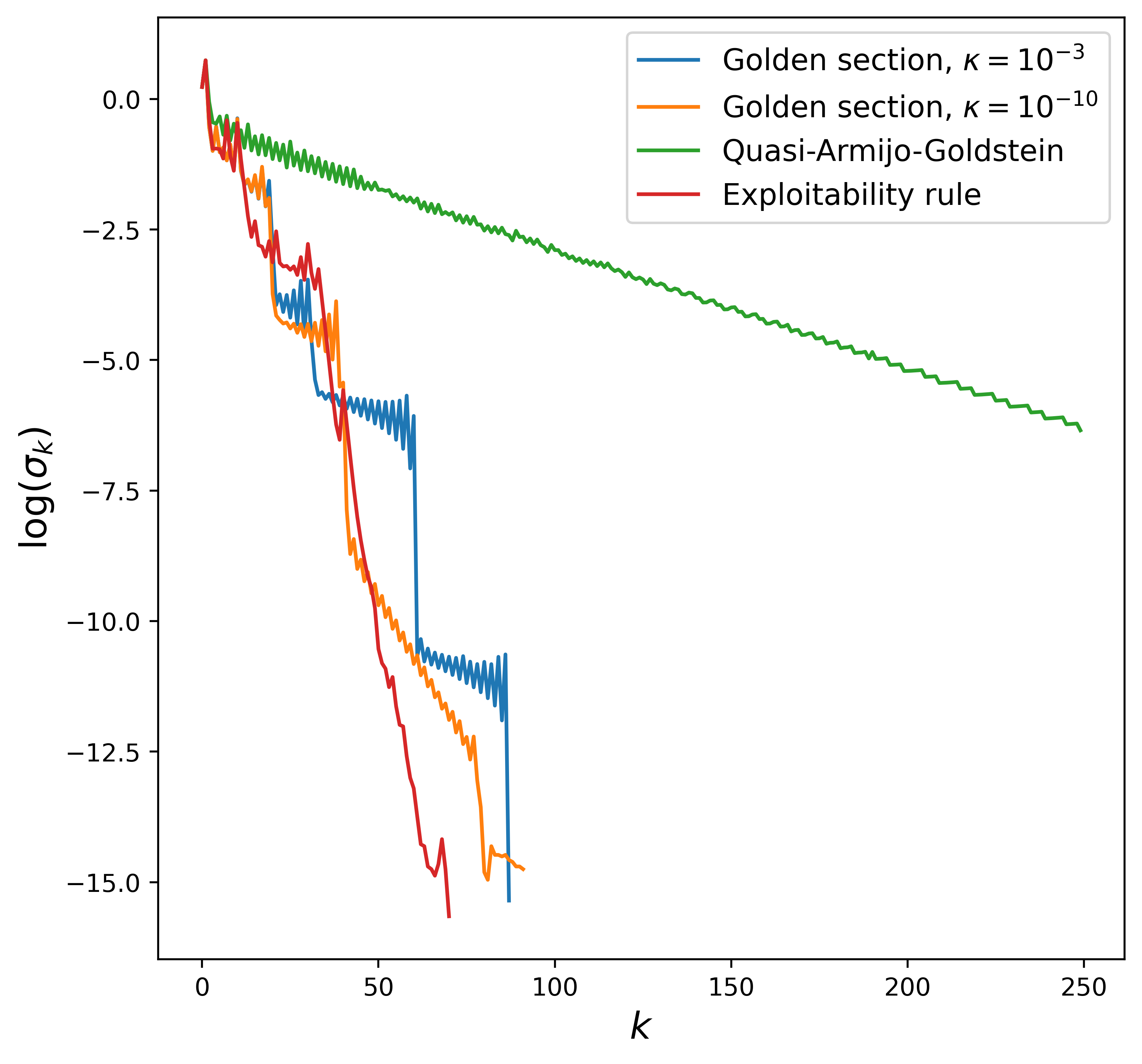}
 \caption{Logarithm of the exploitability for the adaptative stepsizes.}
 \label{fig:closed-loop-exp}
 \end{subfigure}
 \hspace*{\fill}
 \begin{subfigure}{0.45\textwidth}
 \includegraphics[width=0.99\linewidth]{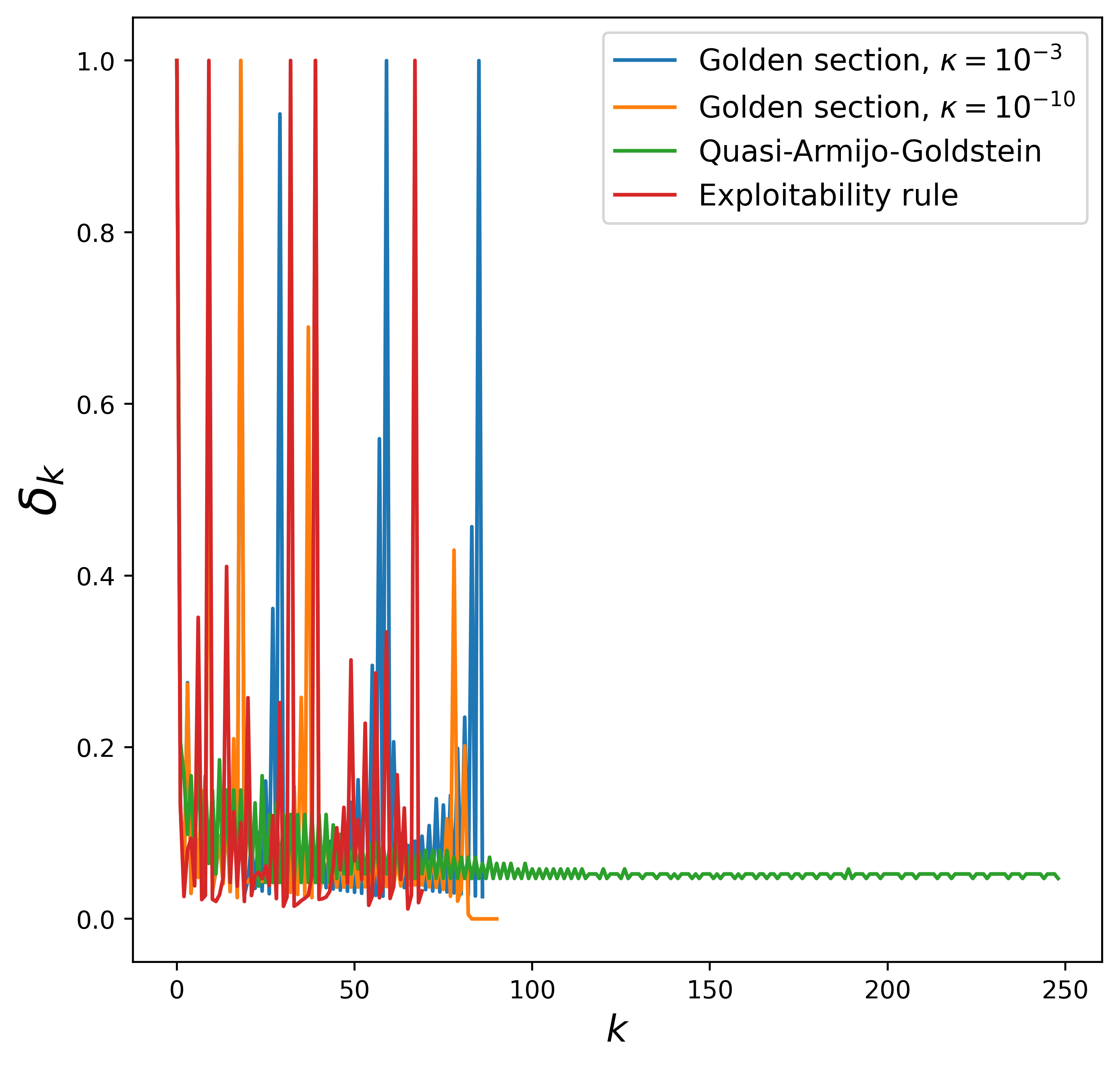}
 \caption{Golden section search, QAG and exploitability-based stepsizes.}
 \label{fig:closed-loop-delta}
 \end{subfigure}
 \end{minipage}
 \caption{Convergence results for prescribed and adaptative stepsizes.}
 \label{fig:convergence}
\end{figure}

The convergence results are reported on Figure \ref{fig:convergence}, for a maximal number $N=250$ of iterations. We use the exploitability as an indicator of convergence, since it can be evaluated explicitely at each iteration. We recall that $\varepsilon_k \leq \sigma_k \leq C \varepsilon_k$.
\begin{itemize}
    \item Prescribed stepsizes, with $\delta_k= p/(k+p)$ and $p \in \{ 1, 2 , 5, 10 \}$.
    Figures \ref{fig:p-exp} and \ref{fig:log-log-p-exp} show that the convergence of the exploitability with an empirical rate of convergence of order $k^{-p}$, in accordance with Lemma \ref{lemma:p_conv}. 
    The value of $p$ which yields the best convergence, for a fixed number of iterations $k$, increases slowly with respect to $k$, in accordance with Remark \ref{remark:cv_open}.
    \item Prescribed stepsizes, with $\delta_k= (k+1)^{-\alpha}$, with $\alpha \in \{ 0.6, 0.7, 0.8, 0.9 \}$. See Figure \ref{fig:alpha-exp}. Similar comments can be done: a better asymptotic rate of convergence is observed for smaller values of $\alpha$; the value of $\alpha$ which yields the best convergence for a fixed number of iterations $k$, decreases with respect to $k$. We have done some tests with smaller values of $\alpha$, which are not shown on the figure. For $N=250$ iterations, the performance is severely degraded for such values of $\alpha$ and convergence cannot be observed in a reasonable number of iterations.
    \item Adaptive stepsizes. See Figure \ref{fig:closed-loop-exp}.
    Optimal stepsizes are approximated with golden section search (see page \pageref{eq:stepsize_rule1bis}) with tolerances $\kappa = 10^{-3}$ and $\kappa = 10^{-10}$. The QAG condition is implemented with $c= 1/4$ and $\tau= 0.9$. The choice of $c$ and $\tau$ is here arbitrary since there is no general rule for the choice of these parameters (see \cite[Chapter 3]{nocedal1999numerical} for further discussions on the topic).
    The three methods all yield a linear convergence; the golden section and the exploitability-based methods are particularly fast. Note that in the case of optimal stepsizes, a very precise resolution of problem \eqref{eq:stepsize_rule1} with tolerance $10^{-10}$ does not improve the convergence of the method, in comparison with the tolerance $10^{-3}$.
\end{itemize}

Finally we compare the time required for the GCG algorithm to satisfy a precision criterion ($\sigma_k \leq 10^{-3}$ and  $\sigma_k \leq 10^{-4}$) for different stepsizes. For the considered example, the time needed to compute the adaptive stepsizes is not significantly longer than the time needed for the resolution of the HJB and the Fokker-Planck equation. Unsurprisingly, the adaptive stepsizes are more efficient than the tested prescribed stepsizes, for the two stopping criteria.

\begin{figure}[h!]
\centering
	\begin{tabular}{|c|c|c|c|}
		\hline
		\multicolumn{2}{|c|}{Learning method} & $\sigma_k \leq 10^{-3}$ & $\sigma_k \leq 10^{-4}$ \\ \hline
		\multicolumn{4}{|c|}{\textbf{Prescribed}} \\ \hline
		\multirow{3}{*}{$\delta_k = \frac{p}{k+p}$} &
		$p = 1$ &  $345.82$ & $3379.41$  \\ \cline{2-4}
		& $p = 2$ & $27.39$  & $85.27$ \\ \cline{2-4}
		& $p = 5$ & $11.14$ & $17.0$ \\ \cline{2-4}
		& $p = 10$ & $13.93$ & $15.61$ \\ \hline
		\multirow{3}{*}{$\delta_k = (k+1)^{-\alpha}$} &
		$\alpha = 0.9$ & $42.46$ & $132.37$ \\ \cline{2-4}
		& $\alpha = 0.8$ & $15.62$ & $32.3$  \\ \cline{2-4}
		& $\alpha = 0.7$ & $9.46$ & $14.95$  \\ \cline{2-4}
		& $\alpha = 0.6$ & $10.3$  & $12.4$  \\ \hline
		\multicolumn{4}{|c|}{\textbf{Adaptative}} \\ \hline
		\multicolumn{2}{|c|}{Quasi-Armijo-Goldstein} &  $9.1$ & $13.11$   \\	\hline
		\multicolumn{2}{|c|}{Golden-section $\kappa = 10^{-3}$} & $2.0$ & $2.2$ \\ \hline
		\multicolumn{2}{|c|}{Golden-section $\kappa = 10^{-10}$} & $2.54$ & $2.68$  \\ \hline
		\multicolumn{2}{|c|}{Exploitability-based} & $3.26$ & $3.91$  \\
		\hline
	\end{tabular}
\caption{Execution time in seconds of the generalized conditional gradient}
\end{figure}

\section{Stability and convergence results} \label{sec:analysis}

This section is dedicated to the analysis of Algorithm \ref{algo:gcg} and is organized in four subsections.
We state in Subsection \ref{subsec:mappings_statement} some technical results concerning the well-posedness of the auxiliary mappings introduced in Subsection \ref{subsec:mappings}.
Subsection \ref{sec:stability} establishes a stability result (Proposition \ref{prop:stability_oc}), necessary to prove the linear speed of convergence. Subsection \ref{sec:convergence} adresses the well-posedness of Algorithm \ref{algo:gcg} and the proof of Theorem \ref{theo:variables}.
Subsection \ref{subsec:proof_main_thm} is dedicated to the proof of Theorem \ref{theo:main}.

\subsection{Well-posedness of the auxiliary mappings}
\label{subsec:mappings_statement}

We provide here technical results, whose proofs can be found in the Appendix. We recall that the sets $\Theta$, $\coupling$, and $\ballR$ have been introduced in \eqref{eq:thetaBig}, \eqref{eq:ball}, and \eqref{eq:ballR}.

\begin{lemma} \label{lemma:fp}
Let $R>0$. Let $v \in \feedback{}$ be such that $\| v \|_{\feedback{}} \leq R$ Then $\fpmap[v]$ is uniquely defined and lies in $W^{2,1,q}(Q)$. Moreover, there exists a constant $C(R)>0$, independent of $v$, such that $\| \fpmap[v] \|_{W^{2,1,q}(Q)} \leq C(R)$. Finally, $\fpmap[v](x,t) \geq 0$, for all $(x,t) \in Q$.
\end{lemma}

\begin{lemma} \label{lemma:fp2}
Let $R>0$. Let $v_1$ and $v_2 \in \feedback{}$. Let $m_i= \fpmap[v_i] \in W^{2,1,q}(Q)$ for $i \in \{1,2 \}$. Assume that $\| v_1 \|_{\feedback{}} \leq R$ and $\| m_2 \|_{L^\infty(Q)} \leq R$. Then, there exists a constant $C(R)$, independent of $v_1$ and $v_2$, such that
\begin{equation*}
\| m_2 - m_1 \|_{L^\infty(0,T;L^2(\mathbb{T}^d))}
\leq
C(R) \Big(
\int_Q |v_2-v_1|^2 m_2 \, \mathrm{d} x \, \mathrm{d} t
\Big)^{1/2}.
\end{equation*}
\end{lemma}

\begin{lemma} \label{lemma:hjb2}
Let $R>0$. There exists a constant $C(R)$ such that for all $(\gamma_1,P_1)$ and $(\gamma_2,P_2)$ in $\ballR$, the following holds:
\begin{equation*}
\| \umap[\gamma_2,P_2] - \umap[\gamma_1,P_1] \|_{L^\infty(Q)}
\leq
C(R) \Big( \| P_2 - P_1 \|_{L^2(0,T;\R^k)}
+ \| \gamma_2 - \gamma_1 \|_{L^\infty(Q)} \Big).
\end{equation*}
\end{lemma}

\begin{proposition} \label{prop:hjb1}
The map $\umap$ is well-defined from $\coupling$ to $W^{2,1,q}(Q)$.
Moreover, for any $R>0$, there exists a constant $C(R)>0$ such that for any $(\gamma,P) \in \ballR$,
$
\| \umap[\gamma,P] \|_{W^{2,1,q}(Q)}
\leq C(R).
$
\end{proposition}

\begin{lemma} \label{lemma:other}
The maps $\vmap$, $\mmap$, and $\wmap$ are well-defined from $\coupling$ to
$\feedback{}$,
$W^{2,1,q}(Q)$, and
$\feedback{}$,
respectively. Moreover, for any $R>0$, there exists $C(R)>0$ such that for any $(\gamma,P) \in \ballR$, it holds:
\begin{equation*}
\| \vmap[\gamma,P] \|_{\feedback{}} +
\| \mmap[\gamma,P] \|_{W^{2,1,q}(Q)}
+
\| \wmap[\gamma,P] \|_{\feedback{}} \leq C(R).
\end{equation*}
\end{lemma}

\begin{lemma} \label{lemma:coupling}
The mappings $\gammamap$ and $\Pmap$ are well-defined. There exists a constant $C>0$ such that for all $m \in W^{2,1,q}(Q)$ and for all $w \in \feedback{}$,
\begin{equation*}
 \|\gamma\|_{W^{1,0,\infty}(Q)} + \| P \|_{L^\infty(0,T;\mathbb{R}^k)} \leq C,
\end{equation*}
where $\gamma= \gammamap[m]$ and $P= \Pmap[w]$.
There exists a constant $C>0$ such that for all $m_1$ and $m_2$ in $W^{2,1,q}(Q)$ and for all $w_1$ and $w_2$ in $\feedback{}$,
\begin{align*}
\| \gammamap[m_2] - \gammamap[m_1] \|_{L^\infty(Q)}
\leq {} &
C \| m_2 - m_1\|_{L^\infty(0,T;L^2(\mathbb{T}^d))}, \\
\| \Pmap[w_2] - \Pmap[w_1] \|_{L^2(0,T;\R^k)} \leq {} &
C \| w_2 - w_1 \|_{L^2(Q;\R^d)}.
\end{align*}
\end{lemma}

\subsection{Stability results for stochastic optimal control problems}
\label{sec:stability}

The main result of this section, Proposition \ref{prop:stability_oc} and its Corollary \ref{coro:stability_mfg}, shows that any approximate solution to Problem \eqref{pb:individual-control-problem-w} is close to its solution, for suitable norms. This is a key result for achieving linear convergence in the GCG method.

\begin{lemma} \label{lemma:quad_growth}
Let $(\hat{\gamma}, \hat{P}) \in \coupling{}$. Let $\hat{m} = \mmap[\hat{\gamma},\hat{P}]$, $\hat{v}= \vmap[\hat{\gamma},\hat{P}]$ and let $\hat{w}= \wmap[\hat{\gamma},\hat{P}]$. There exists a constant $C>0$ such that for any $(m,w) \in \mathcal{R}$, the following holds:
\begin{equation*}
\mathcal{Z}[\hat{\gamma}, \hat{P}](m,w)
- \mathcal{Z}[\hat{\gamma}, \hat{P}](\hat{m},\hat{w})
\geq \frac{1}{C} \int_Q |v(x,t)-\hat{v}(x,t)|^2 m(x,t) \, \dd x \, \dd t,
\end{equation*}
where $v \in L^\infty(Q;\R^d)$ is such that $w= mv$.
\end{lemma}

The proof of Lemma \ref{lemma:quad_growth} relies on the following inequality.

\begin{lemma} \label{lemma:H-L-quad}
Let $(x,t) \in Q$. Let $v \in \R^d$, $\hat{p} \in \R^d$, $m \geq 0$, and $\hat{m} \geq 0$. Let $\hat{v}= - H_p(x,t,\hat{p})$. Let $C>0$ be such that $L(x,t,\cdot)$ is strongly convex with modulus $1/C$. Then,
\begin{equation*}
L(x,t,v) m - L(x,t,\hat{v}) \hat{m} \geq  - H(x,t,p) (m - \hat{m}) 
- \langle \hat{p},w - \hat{w} \rangle + \frac{1}{C}|v - \hat{v}|^2 m.
\end{equation*}
\end{lemma}

\begin{proof}
See \cite[Proof of Proposition 2]{BHP-schauder}.
\end{proof}

\begin{proof}[Proof of Lemma \ref{lemma:quad_growth}]
Using the definition of $\mathcal{Z}[\hat{\gamma},\hat{P}]$, we have
\begin{align}
& \mathcal{Z}[\hat{\gamma},\hat{P}](m,w)
- \mathcal{Z}[\hat{\gamma},\hat{P}](\hat{m},\hat{w}) =
\int_Q \big( L[v]m - L[\hat{v}]\hat{m} \big) \, \dd x \, \dd t \notag \\
& \qquad 
+ \int_Q \hat{\gamma}(m-\hat{m}) \, \dd x \, \dd t
+ \int_Q \langle A^*P, w - \hat{w} \rangle \, \dd x \, \dd t
+ \int_{\mathbb{T}^d} g (m(T)-\hat{m}(T)) \, \dd x. \label{eq:gr1}
\end{align}
We set $\hat{u}= \umap[\hat{\gamma},\hat{P}]$.
By definition of $\hat{v}$, we have $\hat{v} = -  \bm{H}_p[  \nabla \hat{u} +  A^\star \hat{P}]$. Applying Lemma \ref{lemma:H-L-quad} with $\hat{p}= \nabla \hat{u} + A^* \hat{P}$, we obtain
\begin{align*}
& \int_Q \big( L[v]m - L[\hat{v}]\hat{m} \big) \, \dd x \, \dd t
\geq
- \int_{Q} H[\nabla \hat{u} + A^* \hat{P}](m-\bar{m}) \, \dd x \, \dd t \\
& \qquad \quad - \int_Q \langle \nabla \hat{u} + A^* \hat{P}, w - \bar{w} \rangle \, \dd x \, \dd t
+ \frac{1}{C} \int_Q |v-\hat{v}|^2 m \, \dd x \, \dd t .
\end{align*}
We inject the obtained inequality into \eqref{eq:gr1} and we use $\bm{H}[\nabla \hat{u} + A^*\hat{P}] + \hat{\gamma}= -\partial_t \hat{u} - \Delta \hat{u}$. This yields
\begin{align*}
& \mathcal{Z}[\hat{\gamma},\hat{P}](m,w)
- \mathcal{Z}[\hat{\gamma},\hat{P}](\hat{m},\hat{w}) \geq
\int_Q \big( - \partial_t \hat{u} - \Delta \hat{u} \big)\big(m- \hat{m} \big) \, \dd x \, \dd t \\
& \quad 
+ \int_Q \langle -\nabla \hat{u} , w - \hat{w} \rangle \, \dd x \, \dd t
+ \int_{\mathbb{T}^d} g (m(T)-\hat{m}(T)) \, \dd x
+ \frac{1}{C}
\int_Q |v- \hat{v}|^2 m \, \dd x \, \dd t.
\end{align*}
The first three integrals in the right-hand side cancel out: this can be shown by doing an integration by parts and by using the Fokker-Planck equation satisfied by $m$ and $\hat{m}$. This concludes the proof.
\end{proof}

\begin{proposition} \label{prop:stability_oc}
Let $R>0$ and let $(\hat{\gamma},\hat{P}) \in \ballR$. Let $\hat{m}= \mmap[\hat{\gamma},\hat{P}]$ and let $\hat{w}= \wmap[\hat{\gamma},\hat{P}]$.
Let $(m,w) \in \mathcal{R}$ be such that $\| m \|_{L^\infty(Q)} \leq R$.
There exists a constant $C(R)$ such that for any $(m,w) \in \mathcal{R}$,
\begin{align*}
\| m- \hat{m} \|_{L^\infty(0,T;L^2(\mathbb{T}^d))} \leq {} & C(R) \sqrt{\sigma}, \\
\| w- \hat{w} \|_{L^2(Q;\R^d)} \leq {} & C(R) (\sqrt{\sigma} + \sigma),
\end{align*}
where $\sigma= \mathcal{Z}[\hat{\gamma},\hat{P}](m,w)- \mathcal{Z}[\hat{\gamma},\hat{P}](\hat{m},\hat{w})$.
\end{proposition}

As a consequence of Proposition \ref{prop:stability_oc}, $(\mmap[\hat{\gamma},\hat{P}],\wmap[\hat{\gamma},\hat{P}])$ is the unique solution to Problem \eqref{pb:individual-control-problem-w}, with $(\gamma,P)= (\hat{\gamma},\hat{P})$.

\begin{proof}[Proof of Proposition \ref{prop:stability_oc}]
\label{proof:stability_oc}
Lemma \ref{lemma:H-L-quad} yields $\int_Q |v-\hat{v}|^2 m \, \dd x \, \dd t \leq C \sigma$.
By Lemma \ref{lemma:other}, we know that $\|\hat{v}\|_{W^{1,0,\infty}(Q;\R^d)} \leq C$, for some constant depending on $R$. Applying next Lemma \ref{lemma:fp2}, we obtain the estimate of $\| m- \hat{m} \|_{L^\infty(0,T;L^2(\mathbb{T}^d))}$. The estimate of $\| w- \hat{w} \|_{L^2(Q;\R^d)}$ follows directly from
$w-\hat{w}= m(v-\hat{v}) + (m-\hat{m})\hat{v}.
$
The proposition is proved.
\end{proof}

\begin{corollary} \label{coro:stability_mfg}
Let $R>0$. Let $(m,w) \in \mathcal{R}$ be such that $\| m \|_{L^\infty(Q)} \leq R$. There exists a constant $C(R)$ such that
\begin{align*}
\| m- \bar{m} \|_{L^\infty(0,T;L^2(\mathbb{T}^d))} \leq {} & C(R) \sqrt{\varepsilon} \\
\| w- \bar{w} \|_{L^2(Q;\R^d)} \leq {} & C(R) (\sqrt{\varepsilon} + \varepsilon),
\end{align*}
where $\varepsilon= \mathcal{J}(m,w) - \mathcal{J}(\bar{m},\bar{w})$.
\end{corollary}

\begin{proof}
By definition, $\bar{m}= \mmap[\bar{\gamma},\bar{P}]$ and $\bar{w}= \wmap[\bar{\gamma}, \bar{P}]$. We obtain the announced estimates by combining Corollary \ref{coro:lower_bound_full} and Proposition \ref{prop:stability_oc}.
\end{proof}

\subsection{Well-posedness of Algorithm \ref{algo:gcg} and proof of Theorem \ref{theo:variables}.}
\label{sec:convergence}

In this subsection the well-posedness of Algorithm \ref{algo:gcg} is established. The proof of Theorem \ref{theo:variables} is a direct consequence of the well-posedness and the stability results established in the previous subsection. 

\begin{proposition} \label{prop:well-posed}
Algorithm \ref{algo:gcg} generates sequences in the following sets:
$(\bar{m}_k,\bar{w}_k) \in \mathcal{R}$, $(\gamma_k, P_k) \in \coupling{}$, $(u_k,v_k) \in W^{2,1,q}(Q) \times \Theta$, $(m_k,w_k) \in
\mathcal{R}$,
for all $k \in \mathbb{N}$.
Moreover, there exists a constant $C>0$ such that $(\gamma_k,P_k) \in \coupling{}_{C}$ and such that
\begin{align*}
\| u_k \|_{W^{2,1,q}(Q)} +
\| v_k \|_{\feedback{}} \leq {} & C, \\
\| m_k \|_{W^{2,1,q}(Q)}
+
\| w_k \|_{\feedback{}} \leq {} & C, \\
\| \bar{m}_k \|_{W^{2,1,q}(Q)}
+
\| \bar{w}_k \|_{\feedback{}} \leq {} & C,
\end{align*}
for all $k \in \mathbb{N}$.
Finally, there exists a constant $C>0$ such that $\varepsilon_k \leq C$, for all $k \in \mathbb{N}$.
\end{proposition}

\begin{proof}
The well-posedness of the algorithm $\bar{m}_k$, $\bar{w}_k$, $\gamma_k$, $P_k$, $u_k$, $v_k$, $m_k$, and $w_k$ can easily be established by induction.
Lemma \ref{lemma:coupling} shows the existence of a constant $C_0$, independent of $k$ such that $(\gamma_k,P_k) \in \Xi_{C_0}$, for any $k \in \mathbb{N}$. Applying next Proposition \ref{prop:hjb1} and Lemma \ref{lemma:other} with $R=C_0$, we deduce that $\| u_k \|_{W^{2,1,q}(Q)} +
\| v_k \|_{\feedback{}} \leq C$ and
$\| m_k \|_{W^{2,1,q}(Q)}
+
\| w_k \|_{\feedback{}} \leq C$, for some constant $C$ independent of $k$.
The inequality $\| \bar{m}_k \|_{W^{2,1,q}(Q)}
+
\| \bar{w}_k \|_{\feedback{}} \leq C$ follows then immediately from the triangle inequality.
Using the convexity of $\mathcal{R}$ given in Lemma \ref{lemma:convexityR}, we deduce by induction that $(\bar{m}_k,\bar{w}_k) \in \mathcal{R}$, for all $k \in \mathbb{N}$.
The bounds on $m_k$ and $v_k$ imply the existence of a constant $C$ such $\mathcal{J}(m_k,w_k) \leq C$, since
\begin{equation*}
\mathcal{J}(m_k,w_k)
= \int_Q \bm{L} [v_k] m_k \, \dd x \, \dd t
+ \int_0^T \big(\bm{F}[m_k] + \bm{\Phi}[A(m_k v_k)] \big) \dd t + \int_{\mathbb{T}^d} g m_k(T) \, \dd x.
\end{equation*}
The boundedness of $\mathcal{J}(\bar{m}_k,\bar{w}_k)$ follows then by induction, since $\mathcal{J}$ is convex and thus
\begin{equation*}
\mathcal{J}(\bar{m}_{k+1},\bar{w}_{k+1})
\leq
(1- \delta_k) \mathcal{J}(\bar{m}_{k},\bar{w}_{k})
+ \delta_k \mathcal{J}(m_k,w_k).
\end{equation*}
The proposition is proved.
\end{proof}

\paragraph{Proof of Theorem \ref{theo:variables}}
The first estimate concerning $\| \bar{m}_k - \bar{m} \|_{L^\infty(0,T;L^2(\mathbb{T}^d))}$ and $\| \bar{w}_k - \bar{w} \|_{L^2(Q;\R^d)}$ is obtained by combining Corollary \ref{coro:stability_mfg}, the boundedness of $\varepsilon_k$, and the boundedness of $\| m_k \|_{L^\infty(Q)}$.
The second estimate on $\| \gamma_k - \bar{\gamma} \|_{L^\infty(Q)}$ and $\| P_k - \bar{P} \|_{L^\infty(0,T;\R^k)}$ is obtained by combining the first estimate and Lemma \ref{lemma:coupling}.
The third estimate on $\| u_k - \bar{u} \|_{L^{\infty}(Q)}$ follows from the second one and from Lemma \ref{lemma:hjb2}.
Using Proposition \ref{prop:well-posed} and Lemma \ref{lemma:max_reg_embedding}, there exists $C>0$ such that $(\gamma_k,P_k) \in \coupling{}_C$ and $\| \bar{m}_k \|_{L^\infty(Q)} \leq C$, for all $k \in \mathbb{N}$.
Moreover, by construction,
$m_k= \mmap[\gamma_k,P_k]$ and $w_k= \wmap[\gamma_k,P_k]$ and the pair $(\bar{m}_k,\bar{w}_k)$ is $\sigma_k$-optimal for the minimization problem of $\mathcal{Z}[\gamma_k,P_k](\cdot)$. Therefore, Proposition \ref{prop:stability_oc} applies and yields the last estimate that was to be proved.

\subsection{Proof of Theorem \ref{theo:main}}
\label{subsec:proof_main_thm}

In this subsection we prove Theorem \ref{theo:main}, leveraging techniques from \cite{kunisch2022fast}.

\begin{lemma} \label{lemma:linearization}
Let $(m_1,w_1)$ and $(m_2,w_2)$ be in $\mathcal{R}$.  Let $\gamma_1= \gammamap[m_1]$ and let $P_1= \Pmap[w_1]$. Then, there exists a constant $C>0$, independent of $(m_1,w_1)$ and $(m_2,w_2)$ such that
\begin{equation} \label{eq:lower_bound2}
\int_Q \gamma_1 (m_2-m_1) \, \dd x \, \dd t
+ \int_0^T \langle A[w_2-w_1], P_1 \rangle \, \dd t
\leq \mathcal{J}_2(m_2,w_2) - \mathcal{J}_2(m_1,w_1)
\end{equation}
and
\begin{align}
& \mathcal{J}_2(m_2,w_2) - \mathcal{J}_2(m_1,w_1)
\leq
 \int_Q \gamma_1 (m_2-m_1) \, \dd x \, \dd t
+ \int_0^T \langle A[w_2-w_1], P_1 \rangle \, \dd t \notag \\
& \qquad + C_1 \int_0^T \| m_2(t,\cdot)- m_1(t,\cdot) \|_{L^2(\mathbb{T}^d)} \| m_2(t,\cdot)- m_1(t,\cdot) \|_{L^1(\mathbb{T}^d)} \, \dd t \notag \\
& \qquad + C_2 \int_0^T \Big| \int_{\mathbb{T}^d} a(x,t) (w_2(x,t) - w_1(x,t)) \, \dd x \Big|^2 \, \dd t.
\label{eq:upper_bound2}
\end{align}
\end{lemma}

\begin{proof}
Using the definitions of $\gamma_1$ and $P_1$, we obtain
\begin{align*}
& \big( \mathcal{J}_2(m_2,w_2) - \mathcal{J}_2(m_1,w_1) \big)
-
\Big( \int_Q \gamma_1 (m_2-m_1) \, \dd x \, \dd t
+ \int_0^T \langle A[w_2-w_1], P_1 \rangle \, \dd t \Big) \\
& \qquad = (a) + (b),
\end{align*}
where
\begin{align*}
(a) {} & =  \int_0^T \Big( \bm{F}[m_2](t) -  \bm{F}[m_1](t)
- \int_{\mathbb{T}^d} f[m_1](x,t) \big( m_2(x,t)-m_1(x,t) \big) \, \dd x \Big) \, \dd t, \\
(b) {} & = \int_0^T \big( \bm{\Phi}[A w_2]  -  \bm{\Phi}[Aw_1]
- \langle \bm{\phi}[Aw_1](t), Aw_2(t)-A w_1(t) \rangle \big) \,
\dd t.
\end{align*}
Using Assumption \hypF{}, we obtain that
\begin{align*}
(a)
= {} &
\int_0^T \int_0^1 \int_{\mathbb{T}^d}
\big( f[m_1 + s(m_2-m_1)] - f[m_1] \big) ( m_2-m_1 ) \, \dd x \, \dd s \, \dd t.
\end{align*}
Then Assumptions \hypE{} and \hypF{} imply that
\begin{equation*}
0 \leq (a) \leq C_1 \int_0^T \| m_2(t,\cdot)- m_1(t,\cdot) \|_{L^2(\mathbb{T}^d)} \| m_2(t,\cdot)- m_1(t,\cdot) \|_{L^1(\mathbb{T}^d)} \, \dd t.
\end{equation*}
We estimate $(b)$ in a similar way and obtain inequalities \eqref{eq:lower_bound2} and \eqref{eq:upper_bound2} easily.
\end{proof}

\begin{corollary} \label{coro:lower_bound_full}
Let $(m_1,w_1)$ and $(m_2,w_2)$ be in $\mathcal{R}$.  Let $\gamma_1= \gammamap[m_1]$ and let $P_1= \Pmap[w_1]$.
Then,
\begin{equation*}
\mathcal{Z}[\gamma_1,P_1](m_2,w_2)
- \mathcal{Z}[\gamma_1,P_1](m_1,w_1)
\leq
\mathcal{J}(m_2,w_2) - \mathcal{J}(m_1,w_1).
\end{equation*}
\end{corollary}

\begin{proof}
This is an immediate consequence of inequality \eqref{eq:lower_bound2} from Lemma \ref{lemma:linearization} and the definitions of $\mathcal{J}$ and $\mathcal{Z}$.
\end{proof}

\begin{lemma} \label{eq:seq_gamma}
For all $k \in \mathbb{N}$, it holds that $\varepsilon_k \leq \sigma_k$.
\end{lemma}

\begin{proof}
By Corollary \ref{coro:lower_bound_full}, we have for all $(m,w) \in \mathcal{R}$
\begin{equation*}
\mathcal{Z}[\gamma_k,P_k](\bar{m}_k,\bar{w}_k)
- \mathcal{Z}[\gamma_k,P_k](m,w)
\leq
\mathcal{J}(\bar{m}_k,\bar{w}_k)
- \mathcal{J}(m,w) \\
\leq \varepsilon_k.
\end{equation*}
By definition, $\sigma_k$ is the supremum of the left-hand side with respect to $(m,w)$. The conclusion follows immediately.
\end{proof}

\begin{lemma} \label{lemma:upper_bound}
Let $C_1>0$ and $C_2>0$ denote the Lipschitz constants of $f$ and $\phi$ (see Assumption \hypE{}). Then for any $\delta \in [0,1]$, it holds that
\begin{equation} \label{eq:upper_bound1234}
\mathcal{J}(\bar{m}_k^\delta,\bar{w}_k^\delta) \leq
\mathcal{J}(\bar{m}_{k},\bar{w}_{k}) - \delta \sigma_k + \big( C_1 D_k^{(1)} + C_2 D_k^{(2)} \big) \delta^2,
\end{equation}
where $(\bar{m}_k^{\delta},\bar{w}_k^{\delta})$ is defined by \eqref{eq:def_conv_k} and where $D_k^{(1)}$ and $D_k^{(2)}$ are defined by \eqref{eq:D_k}.
Moreover, there exists a constant $C>0$ such that
\begin{equation} \label{eq:upper_bound_bis}
\mathcal{J}(\bar{m}_k^\delta,\bar{w}_k^\delta) \leq
\mathcal{J}(\bar{m}_{k},\bar{w}_{k}) - \delta \sigma_k + C \sigma_k \delta_k^2.
\end{equation}
\end{lemma}

Recall that the terms $D_k^{(1)}$ and $D_k^{(2)}$ represent the distance of the current approximate solution to the solution of the (partially) linearized problem. In the classical proof of convergence of the Frank-Wolfe algorithm, one writes a similar estimate to \eqref{eq:upper_bound1234}, where $D_k^{(1)}$ and $D_k^{(2)}$ are simply bounded by some constant. In order to achieve linear convergence (instead of the classical sublinear rate of convergence), it is crucial to keep these terms and to estimate them with the exploitability $\sigma_k$, as will become clear in the next section. This proof technique is largely inspired by \cite[Section 4.3]{kunisch2022fast}.

\begin{proof}[Proof of Lemma \ref{lemma:upper_bound}]
The proof relies on the decomposition $\mathcal{J}= \mathcal{J}_1 + \mathcal{J}_2$ introduced in \eqref{eq:cost_decomposition}.
First, by convexity of $\tilde{L}$, we have
\begin{equation}
\mathcal{J}_1(\bar{m}_k^{\delta},\bar{w}_k^{\delta}) -
\mathcal{J}_1(\bar{m}_k,\bar{w}_k)
\leq \delta
\big( \mathcal{J}_1(m_k,w_k) - \mathcal{J}_1(\bar{m}_k,\bar{w}_k) \big). \label{eq:bound_dev1}
\end{equation}
Next, using inequality \eqref{eq:upper_bound2} of Lemma \ref{lemma:linearization}, we obtain that
\begin{align}
& \mathcal{J}_2(\bar{m}_k^{\delta},\bar{w}_k^{\delta}) -
\mathcal{J}_2(\bar{m}_k,\bar{w}_k)
\leq \delta \int_Q \gamma_k (m_k-\bar{m}_k) \, \dd x \, \dd t \notag \notag \\
& \quad + \delta \int_0^T \langle A[w_k-\bar{w}_k], P_k \rangle \, \dd t
+ C \delta^2 \Big( \| m_k- \bar{m}_k \|_{L^2(Q)}^2 + \| w_k- \bar{w}_k \|_{L^2(Q;\R^d)}^2 \Big).
\label{eq:bound_dev2}
\end{align}
By definition of $\mathcal{Z}[\gamma_k,P_k]$, we have
\begin{align}
\sigma_k
= {} & \mathcal{Z}[\gamma_k,P_k](\bar{m}_k,\bar{w}_k)-\mathcal{Z}[\gamma_k,P_k](m_k,w_k) \notag \\[0.5em]
= {} &
\mathcal{J}_1(m_k,w_k) - \mathcal{J}_1(\bar{m}_k,\bar{w}_k) \notag \\
& \quad + \int_Q \gamma_k (m_k-\bar{m}_k) \, \dd x \, \dd t
+ \int_0^T \langle A[w_k-\bar{w}_k], P_k \rangle \, \dd t. \label{eq:zrelation}
\end{align}
Finally, by Theorem \ref{theo:variables} we have
\begin{equation} \label{eq:boundsigma}
\| m_k - \bar{m}_k \|_{L^\infty(0,T;L^2(\mathbb{T}^d))}^2
+
\| w_k - \bar{w}_k \|_{L^2(Q;\R^d)}^2 \leq C \sigma_k.
\end{equation}
Summing up \eqref{eq:bound_dev1} and \eqref{eq:bound_dev2} and combining the result with \eqref{eq:zrelation} and \eqref{eq:boundsigma}, we obtain the announced result.
\end{proof}

\begin{lemma} \label{lemma:exploitability}
There exists $C>0$ such that $\sigma_k \leq C \varepsilon_k$, for all $k \in \mathbb{N}$.
\end{lemma}

\begin{proof}
For any $\delta \in [0,1]$, Lemma \ref{lemma:upper_bound} yields
$\mathcal{J}(\bar{m},\bar{w}) \leq \mathcal{J}(\bar{m}_{k},\bar{w}_{k}) - \delta \sigma_k + C \delta^2 \sigma_k$.
It follows that
\begin{equation}
0 \leq \varepsilon_k - (1 - C \delta) \delta \sigma_k.
\end{equation}
Increasing if necessary the value of $C$, we can assume that $C \geq 1/2$. Taking $\delta= 1/(2C)$, we conclude that $0 \leq \varepsilon - 1/(4C) \sigma_k$, as was to be proved.
\end{proof}

\paragraph{Proof of Theorem \ref{theo:main}}
\label{subsec:proof_main}

 We start with the proof of the second part of the theorem, which is almost direct.
Combining the upper bound of the cost function proved in Lemma \ref{lemma:upper_bound} (inequality \eqref{eq:upper_bound_bis}) with the inequality $\varepsilon_k \leq \sigma_k$ and the bound on the exploitability obtained in Lemma \ref{lemma:exploitability}, we obtain the following inequality:
\begin{equation*}
\varepsilon_{j+1} \leq
\big( 1 - \delta_j + C \delta_j^2 \big) \varepsilon_j, \quad \forall j \in \mathbb{N}.
\end{equation*}
Here the constant $C$ is independent of the choice of stepsize.
Multiplying the obtained inequalities for $j=0,\ldots,k$, we obtain
\begin{equation*}
\varepsilon_{j+1}
\leq \varepsilon_0
\exp \Big(
\sum_{j=0}^k \, \ln \left( 1 - \delta_j + C \delta_j^2 \right) \Big).
\end{equation*}
Using next the inequality $\ln(x) \leq x-1$, satisfied for any $x>0$, we obtain the desired result, inequality \eqref{eq:conv_general}.

Let us consider the case of adaptive stepsizes. Let us fix the iteration number $k$. It suffices to show that for $\delta_k$ satisfying either \eqref{eq:stepsize_rule1}, \eqref{eq:stepsize_rule2}, or \eqref{eq:stepsize_rule3}, there exists a constant $\beta \in (0,1)$, independent of $k$, such that $\varepsilon_{k+1} \leq \beta \varepsilon_k$.

\emph{Case of the QAG condition.} The main idea is to show that the QAG condition is satisfied when $\delta$ is smaller than a certain threshold, which is independent of $k$.
Let $\delta > 0$ be such that the condition is not satisfied. Then, using Lemma \ref{lemma:upper_bound},
\begin{equation*}
\mathcal{J}(\bar{m},\bar{w}) - c \delta \sigma_k
<
\mathcal{J}(\bar{m}_k^{\delta},\bar{w}_k^{\delta})
\leq
\mathcal{J}(\bar{m},\bar{w})  - \delta \sigma_k + C \delta^2 \sigma_k.
\end{equation*}
Re-arranging, we deduce that $\delta > \bar{\delta}:=\frac{(1-c)}{C} > 0$. By contraposition, the QAG condition holds true for any $\delta$ such that $\delta \leq \bar{\delta}$.

We can now prove that $i_k$ is finite and uniformly bounded. Two cases can be considered. If $\bar{\delta}\geq 1$, then $i_k= 0$ and $\delta_k= 1$. Otherwise, for any $j \in \mathbb{N}$, we have
\begin{equation*}
\tau^j \leq \bar{\delta}
\Longleftrightarrow
j \geq \mathrm{ln}(\bar{\delta})/\mathrm{ln}(\tau).
\end{equation*}
Therefore, $i_k \leq \lceil \mathrm{ln}(\bar{\delta}) /\mathrm{ln}(\tau) \rceil$. If $i_k=0$, then $\delta_k= 1$. Otherwise, if $i_k > 0$, then $\tau^{i_k-1}$ does not satisfy the QAG condition and therefore, $\tau^{i_k-i} \geq \bar{\delta}$ and thus $\delta_k= \tau^{i_k} \geq \tau \bar{\delta}$.
So, in all cases, we have $\delta_k \geq \delta_{\min} := \min(1, \tau \bar{\delta})$. It follows that
\begin{align}
\varepsilon_{k+1}
= {} &
\mathcal{J}(\bar{m}_{k+1},\bar{w}_{k+1})
- \mathcal{J}(\bar{m},\bar{w}) \notag \\
\leq {} &
\big(
\mathcal{J}(\bar{m}_{k},\bar{w}_{k})
- \mathcal{J}(\bar{m},\bar{w})
\big)
- c \sigma_k \delta_k \notag \\
\leq {} &
\varepsilon_k - c \delta_{\min} \sigma_k \leq (1- c\delta_{\min}) \varepsilon_k. \label{eq:cv_qag_lin}
\end{align}
The last inequality was obtained with Lemma \ref{eq:seq_gamma}.

\emph{Case of exploitability-based stepsizes}.
Let us set $a_k= \big(C_1 D_k^{(1)} + C_2 D_k^{(2)} \big)$. By definition, $\delta_k= \min \big( 1, \frac{\sigma_k}{2a_k} \big)$.
Assume that $\sigma_k \geq 2a_k$, i.e.\@ $a_k \leq \sigma_k/2$. Then $\delta_k=1$. Inequality \eqref{eq:upper_bound1234} in Lemma \ref{lemma:upper_bound} and Lemma \ref{eq:seq_gamma} yield
\begin{equation*}
\varepsilon_{k+1} \leq \varepsilon_k - \frac{\sigma_k}{2} \leq \frac{1}{2} \varepsilon_k.
\end{equation*}
Now assume that $\sigma_k < 2a_k$. Then $\delta_k= \frac{\sigma_k}{2a_k}$. By \eqref{eq:upper_bound1234}, $
\varepsilon_{k+1} \leq
\varepsilon_k - \frac{\sigma_k^2}{4 a_k}. 
$
It follows from the last estimate of Theorem \ref{theo:variables} that $a_k \leq C \sigma_k$. Therefore, by Lemma \ref{eq:seq_gamma},
\begin{equation*}
\varepsilon_{k+1} \leq
\varepsilon_k - \frac{\sigma_k}{4C} \leq \Big( 1- \frac{1}{4C} \Big) \varepsilon_k, 
\end{equation*}
as was to be proved.

\emph{Case of an optimal stepsize.} If $\delta_k$ minimizes $\mathcal{J}(\bar{m}_k^\delta,\bar{w}_k^\delta)$, then \eqref{eq:cv_qag_lin} is necessarily satisfied. This concludes the proof of the theorem.

\section{Conclusion}
\label{sec:conclusion}

The connection between the GCG method and fictitious play investigated in this article is not specific to second-order MFGs and could be established in different settings. 
Before discussing some possible extensions of our work, let us make some general comments.
\begin{itemize}
\item While we have focused here on a class of MFGs with a convex potential formulation, the case of nonconvex potential MFGs is also of interest. An example of interactions with a nonconvex variational structure, arising from a consensus model, is given in  \cite{santambrogio2021cucker}. In this situation the GCG algorithm might not converge. In \cite{hadikhanloo2017learning}, the authors only show that any convergent sub-sequence generated by the fictitious play converges to an equilibrium. This problem can be tackled, assuming that solutions are stable \cite{briani2018stable}. Also very few results are available concerning the conditional gradient in the nonconvex setting. The article \cite{lacoste2016convergence} shows that in the case of nonconvex optimization problems, the Frank-Wolfe algorithm converges to stationary points, when suitable stepsizes are utilized. The notion of stationarity involved in \cite{lacoste2016convergence} should lead to the MFG system associated with the nonconvex potential formulation.
\item The very first assumption to be satisfied in the analysis of the Frank-Wolfe algorithm is the Lipschitz continuity of the gradient of the cost function. In the case of MFGs, this means that the coupling functions should Lipschitz-continuous in suitable functional spaces. As a consequence, the analysis of the GCG method for potential MFGs with nonsmooth coupling functions (for example, MFGs with a local congestion term) may be particularly difficult.
\item In general, the GCG method only has a sublinear rate of convergence. The linear rate of convergence obtained in this article heavily relies on the specific stability analysis which was done in Section \ref{sec:stability} for optimal control problems.
\end{itemize}
A first natural extension of our work concerns the case of an unbounded domain. We expect that the linear convergence can be achieved. At a technical level, one difficulty concerns the boundedness of the distribution $m_k$. While we have established it with the help of parabolic estimates, we could follow the methodology of \cite{cardaliaguet2018mean} to address this more general case.
We also think that linear convergence of the GCG method can be established for fully discrete MFGs, as formulated in \cite{bonnans2021discrete}, taking Lipschitz-continuous coupling functions and a strongly convex running cost.
Finally, we mention the case of first-order MFGs, in a Lagrangian formulation (as formulated in \cite{cardaliaguet2017learning}, for example): for this case, we only expect a sublinear rate of convergence for the GCG method.

Finally, let us mention that the fictitious play algorithm is quite similar to the policy iteration method proposed and analyzed in \cite{cacace2021policy} and \cite{camilli2022rates} for MFGs, since this method also relies on iterative resolutions of the HJB and the Fokker-Planck equation. The analysis techniques of our article may bring new insights to the policy iteration method.

\bibliographystyle{plain}
\bibliography{biblio}

\begin{thebibliography}{10}

\bibitem{achdou2020mean}
Y.~Achdou and M.~Lauri{\`e}re.
\newblock Mean field games and applications: Numerical aspects.
\newblock {\em Mean Field Games: Cetraro, Italy 2019}, pages 249--307, 2020.

\bibitem{Benamou2015}
J.-D. Benamou and G.~Carlier.
\newblock Augmented {L}agrangian methods for transport optimization, mean field
  games and degenerate elliptic equations.
\newblock {\em Journal of Optimization Theory and Applications}, 167(1):1--26,
  Oct 2015.

\bibitem{benamou2019entropy}
J.-D. Benamou, G.~Carlier, S.~Di~Marino, and L.~Nenna.
\newblock An entropy minimization approach to second-order variational
  mean-field games.
\newblock {\em Mathematical Models and Methods in Applied Sciences},
  29(08):1553--1583, 2019.

\bibitem{benamou2017variational}
J.-D. Benamou, G.~Carlier, and F.~Santambrogio.
\newblock Variational mean field games.
\newblock In {\em Active Particles, Volume 1}, pages 141--171. Springer, 2017.

\bibitem{BHP-schauder}
J.F. Bonnans, S.~Hadikhanloo, and L.~Pfeiffer.
\newblock Schauder estimates for a class of potential mean field games of
  controls.
\newblock {\em Appl.\@ Math.\@ Optim.}, 83:1431--1464, 2021.

\bibitem{bonnans2021discrete}
J.F. Bonnans, P.~Lavigne, and L.~Pfeiffer.
\newblock Discrete potential mean field games: duality and numerical
  resolution.
\newblock {\em Mathematical Programming}, pages 1--38, 2023.

\bibitem{bredies2009generalized}
K.~Bredies, D.A. Lorenz, and P.~Maass.
\newblock A generalized conditional gradient method and its connection to an
  iterative shrinkage method.
\newblock {\em Computational Optimization and Applications}, 42(2):173--193,
  2009.

\bibitem{briani2018stable}
A.~Briani and P~Cardaliaguet.
\newblock Stable solutions in potential mean field game systems.
\newblock {\em Nonlinear Differential Equations Appl.}, 25(1):1, 2018.

\bibitem{briceno2019implementation}
L.~Brice{\~n}o-Arias, D.~Kalise, Z.~Kobeissi, M.~Lauri{\`e}re, A.M.
  Gonz{\'a}lez, and F.~Silva.
\newblock On the implementation of a primal-dual algorithm for second order
  time-dependent mean field games with local couplings.
\newblock {\em ESAIM: Proceedings and Surveys}, 65:330--348, 2019.

\bibitem{briceno2018proximal}
L.~Brice{\~n}o-Arias, D.~Kalise, and F.~Silva.
\newblock Proximal methods for stationary mean field games with local
  couplings.
\newblock {\em SIAM J.\@ Control Optim.}, 56(2):801--836, 2018.

\bibitem{cacace2021policy}
S.~Cacace, F.~Camilli, and A.~Goffi.
\newblock A policy iteration method for mean field games.
\newblock {\em ESAIM: Control, Optimisation and Calculus of Variations}, 27:85,
  2021.

\bibitem{camilli2022rates}
F.~Camilli and Q.~Tang.
\newblock Rates of convergence for the policy iteration method for mean field
  games systems.
\newblock {\em Journal of Mathematical Analysis and Applications},
  512(1):126--138, 2022.

\bibitem{cardaliaguet2015second}
P.~Cardaliaguet, P.J. Graber, A.~Porretta, and D.~Tonon.
\newblock Second order mean field games with degenerate diffusion and local
  coupling.
\newblock {\em Nonlinear Differential Equations Appl.}, 22(5):1287--1317, 2015.

\bibitem{cardaliaguet2017learning}
P.~Cardaliaguet and S.~Hadikhanloo.
\newblock Learning in mean field games: the fictitious play.
\newblock {\em ESAIM:COCV}, 23(2):569--591, 2017.

\bibitem{cardaliaguet2018mean}
P.~Cardaliaguet and C.-A. Lehalle.
\newblock Mean field game of controls and an application to trade crowding.
\newblock {\em Math.\@ Financ.\@ Econ.}, 12(3):335--363, 2018.

\bibitem{elie2019approximate}
R.~Elie, J.~P{\'e}rolat, M.~Lauri{\`e}re, M.~Geist, and O.~Pietquin.
\newblock Approximate fictitious play for mean field games.
\newblock {\em arXiv preprint arXiv:1907.02633}, 2019.

\bibitem{fleming2006controlled}
W.H. Fleming and H.M. Soner.
\newblock {\em Controlled {M}arkov processes and viscosity solutions},
  volume~25.
\newblock Springer Science \& Business Media, 2006.

\bibitem{geist2021concave}
M.~Geist, J.~P{\'e}rolat, M.~Lauri{\`e}re, R.~Elie, S.~Perrin, O.~Bachem,
  R.~Munos, and O.~Pietquin.
\newblock Concave utility reinforcement learning: the mean-field game
  viewpoint.
\newblock {\em arXiv preprint arXiv:2106.03787}, 2021.

\bibitem{gilbarg2015elliptic}
D.~Gilbarg and N.S. Trudinger.
\newblock {\em Elliptic partial differential equations of second order}.
\newblock Springer, 2015.

\bibitem{graber2018variational}
P.J. Graber and C.~Mouzouni.
\newblock Variational mean field games for market competition.
\newblock In {\em PDE models for multi-agent phenomena}, pages 93--114.
  Springer, 2018.

\bibitem{graber2020weak}
P.J. Graber, A.~Mullenix, and L.~Pfeiffer.
\newblock Weak solutions for potential mean field games of controls.
\newblock {\em Nonlinear Differential Equations Appl.}, 28(5):1--34, 2021.

\bibitem{hadikhanloo2017learning}
S.~Hadikhanloo.
\newblock Learning in anonymous nonatomic games with applications to
  first-order mean field games.
\newblock {\em arXiv preprint arXiv:1704.00378}, 2017.

\bibitem{HADIKHANLOO2019369}
S.~Hadikhanloo and F.J. Silva.
\newblock Finite mean field games: Fictitious play and convergence to a first
  order continuous mean field game.
\newblock {\em J.\@ Math.\@ Pures Appl.}, 132:369 -- 397, 2019.

\bibitem{HCMieeeAC06}
M.~Huang, R.P. Malham{\'e}, and P.E. Caines.
\newblock Large population stochastic dynamic games: closed-loop
  {M}c{K}ean-{V}lasov systems and the {N}ash certainty equivalence principle.
\newblock {\em Communications in Information \& Systems}, 6(3):221--252, 2006.

\bibitem{jaggi2013revisiting}
M.~Jaggi.
\newblock Revisiting {F}rank-{W}olfe: Projection-free sparse convex
  optimization.
\newblock In {\em International Conference on Machine Learning}, pages
  427--435. PMLR, 2013.

\bibitem{kunisch2022fast}
K.~Kunisch and D.~Walter.
\newblock On fast convergence rates for generalized conditional gradient
  methods with backtracking stepsize.
\newblock {\em Numerical Algebra, Control and Optimization}, pages 0--0, 2022.

\bibitem{lacoste2016convergence}
S.~Lacoste-Julien.
\newblock Convergence rate of frank-wolfe for non-convex objectives.
\newblock {\em arXiv preprint arXiv:1607.00345}, 2016.

\bibitem{LSU}
O.A. Ladyzhenskaia, V.A. Solonnikov, and N.N. Ural'tseva.
\newblock {\em Linear and quasi-linear equations of parabolic type}, volume~23.
\newblock American Math.\@ Soc., 1988.

\bibitem{LL07mf}
J.-M. Lasry and P.-L. Lions.
\newblock Mean field games.
\newblock {\em Japanese journal of mathematics}, 2(1):229--260, 2007.

\bibitem{lions1971optimal}
J.-L. Lions.
\newblock {\em Optimal control of systems governed by partial differential
  equations}, volume 170.
\newblock Springer Verlag, 1971.

\bibitem{nocedal1999numerical}
J.~Nocedal and S.J. Wright.
\newblock {\em Numerical optimization}.
\newblock Springer, 1999.

\bibitem{perrin2020fictitious}
S.~Perrin, J.~P{\'e}rolat, M.~Lauri{\`e}re, M.~Geist, R.~Elie, and O.~Pietquin.
\newblock Fictitious play for mean field games: Continuous time analysis and
  applications.
\newblock {\em arXiv preprint arXiv:2007.03458}, 2020.

\bibitem{santambrogio2021cucker}
F.~Santambrogio and W.~Shim.
\newblock A cucker--smale inspired deterministic mean field game with velocity
  interactions.
\newblock {\em SIAM J.\@ Control Optim.}, 59(6):4155--4187, 2021.

\bibitem{sorin2022continuous}
S.~Sorin.
\newblock Continuous time learning algorithms in optimization and game theory.
\newblock {\em Dynamic Games and Applications}, pages 1--22, 2022.

\bibitem{wang2021global}
W.~Wang, J.~Han, Z.~Yang, and Z.~Wang.
\newblock Global convergence of policy gradient for linear-quadratic mean-field
  control/game in continuous time.
\newblock In {\em International Conference on Machine Learning}, pages
  10772--10782. PMLR, 2021.

\end{thebibliography}

\section*{Statement and Declarations}

This work was supported by a public grant as part of the
Investissement d'avenir project, reference ANR-11-LABX-0056-LMH,
LabEx LMH. The authors have no relevant financial or non-financial interests to disclose.

\appendix

\section{Appendix: Regularity of the auxiliary mappings}
\label{sec:mappings}

This appendix contains the proofs of the technical lemmas of Subsection \ref{subsec:mappings_statement}.

\subsection{Parabolic estimates}

In this section we provide estimates for the following parabolic equation:
\begin{equation}
  \label{eq:parabolic}
\begin{array}{rlr}
\partial_t u - \sigma \Delta u + \langle b, \nabla u \rangle+ c u \, = & \!\!\! h, \quad & (x,t) \in Q, \\
u(x,0) \,= & \!\!\! u_0(x), & x \in \mathbb{T}^d,
\end{array}
\end{equation}
for different assumptions on $b \colon Q \rightarrow \R^d$, $c \colon Q \rightarrow \R$, $h \colon \rightarrow \R$, and $u_0 \colon \mathbb{T}^d \rightarrow \R$. The proofs of the following results can be found in the Appendix of \cite{BHP-schauder}; they largely rely on \cite{LSU}. We recall that $q$ is a fixed parameter and $q>d+2$. 

In the next theorem, we consider the Sobolev space $W^{2-2/p,p}(\mathbb{T}^d)$ with a fractional order of derivation, see \cite[section II.2]{LSU} for a definition.

\begin{theorem}
\label{theo:max_reg1}
 For all $R>0$, there exists $C>0$ such that for all $u_0 \in W^{2-2/q,q}(\mathbb{T}^d)$, for all $b \in L^q(Q;\R^d)$, for all $c \in L^q(Q)$, and for all $h \in L^q(Q)$ satisfying
\begin{equation*}
\| u_0 \|_{W^{2-2/q,q}(\mathbb{T}^d)} +
\| b \|_{L^q(Q;\R^d)} +
\| c \|_{L^q(Q)} +
\| h \|_{L^q(Q)} \leq R,
\end{equation*}
equation \eqref{eq:parabolic} has a unique solution $u$ in $W^{2,1,q}(Q)$. Moreover, $\| u \|_{W^{2,1,q}(Q)} \leq C$.
\end{theorem}

\begin{theorem} \label{thm:bound-u-u0-h}
There exists $C>0$ such that for all $u_0 \in W^{2-2/q,q}(\mathbb{T}^d)$ and for all $h \in L^q(Q)$, the unique solution $u$ to \eqref{eq:parabolic} (with $b = 0$ and $c=0$) satisfies the following estimate:
\begin{equation} \nonumber
\|u\|_{W^{2,1,q}(Q)} \leq C \big( \|u_0\|_{W^{2-2/q,q}(\mathbb{T}^d)} + \|h\|_{L^q(Q)} \big). 
\end{equation}
\end{theorem}

\begin{theorem}
\label{theo:holder_reg_classical}
For all $\beta \in (0,1)$, for all $R>0$, there exist $\alpha \in (0,1)$ and $C>0$ such that for all $u_0 \in \mathcal{C}^{2+ \beta}(\mathbb{T}^d)$, $b \in \mathcal{C}^{\beta,\beta/2}(Q;\R^d)$, $c \in \mathcal{C}^{\beta,\beta/2}(Q)$ and $h \in \mathcal{C}^{\beta,\beta/2}(Q)$ satisfying
$\| u_0 \|_{\mathcal{C}^{2+ \beta}(\mathbb{T}^d)} +
\| b \|_{\mathcal{C}^{\beta,\beta/2}(Q;\R^d)} +
\| c \|_{\mathcal{C}^{\beta,\beta/2}(Q)} +
\| h \|_{\mathcal{C}^{\beta,\beta/2}(Q)} \leq R$,
the solution to \eqref{eq:parabolic} lies in $\mathcal{C}^{2+\alpha,1+\alpha/2}(Q)$ and satisfies
$\| u \|_{\mathcal{C}^{2+\alpha,1+\alpha/2}(Q)} \leq C$.
\end{theorem}

\subsection{Fokker-Planck equation} \label{subsec:43}

\begin{proof}[Proof of Lemma \ref{lemma:fp}] \label{proof:fp}
Let us write the Fokker-Planck equation in the form of equation \eqref{eq:parabolic}: $\partial_t m - \Delta m + (\nabla \cdot v) m + \langle v, \nabla m \rangle = 0$.
The first part of lemma follows from Theorem \ref{theo:max_reg1}. The nonnegativity of $\fpmap[v]$ is proved in \cite[Lemma 3]{BHP-schauder}.
\end{proof}

\begin{proof}[Proof of Lemma \ref{lemma:fp2}] \label{proof:fp2}
Set $w= v_2-v_1$ and $\mu= m_2-m_1$. Then $\mu$ is the solution to
\begin{equation} \nonumber
\begin{array}{rlr}
\partial_t \mu  - \Delta \mu  + \nabla \cdot ( v_1 \mu ) \, = & \!\!\! -  \nabla \cdot (w  m_2), \quad & (x,t) \in Q, \\
\mu(x,0) \, = &\!\!\! 0, & x \in \mathbb{T}^d.
\end{array}
\end{equation}
Set $V=W^{1,2}(\mathbb{T}^d)$ and consider the Gelfand triple $(V,L^2(\mathbb{T}^d),V^*)$, where $V^*$ denotes the dual of $V$.
Then $\mu$ is solution of a parabolic equation of the form
\begin{equation} \nonumber
\begin{array}{rlr}
 \partial_t  m(t) + B(t) m(t) \, = & \!\!\! f(t), \quad & (x,t) \in Q, \\
m(x,0) \, = & \!\!\! 0, & x \in \mathbb{T}^d,
\end{array}
\end{equation}
where $B(t) \in L(V, V^{*})$ and $f(t) \in V^{*}$. For any $m \in V$, we have
\begin{align*} \nonumber
\langle B(t) m, m \rangle_{V} & = \int_{\mathbb{T}^d} \big( - \Delta m + \nabla \cdot v_1(t) m + \langle v_1(t), \nabla m \rangle \big) m \, \dd x \\
& = \int_{\mathbb{T}^d} \big( |\nabla m|^2 - \langle v_1(t), \nabla m \rangle m \big) \dd x,
\end{align*}
where the second equality is obtained by integration by parts.
Using Cauchy-Schwarz inequality and $\| v_1 \|_{L^\infty(Q;\R^d)} \leq R$, we obtain the following inequality:
\begin{equation*}
\langle B(t) m, m \rangle_{V}
\geq \| \nabla m \|_{L^2(\mathbb{T}^d;\R^d)}^2
- C \| \nabla m \|_{L^2(\mathbb{T}^d;\R^d)} \| m \|_{L^2(\mathbb{T}^d)}
\end{equation*}
where the constant $C$ is independent of $t$ (but depends on $R$).
A direct application of Young's inequality yields the existence of $C$ (depending on $R$) such that
\begin{equation} \nonumber
\langle B(t) m, m \rangle_{V} \geq \frac{1}{2} \|m \|^2_{V} - C\|m \|^2_{L^2(\mathbb{T}^d)}. 
\end{equation}
Thus $B(t)$ is semi-coercive, uniformly in time. With similar techniques, one can show that
$\langle B(t)m,m' \rangle_V \leq C \| m \|_V \| m' \|_V$,
for a.e.\@ $t \in (0,T)$ and for all $m$ and $m'$ in $V$.
We can apply \cite[Chapter 3, Theorems 1.1 and 1.2]{lions1971optimal}, from which we derive
\begin{align*}
\| \mu \|_{L^\infty(0,T;L^2(\mathbb{T}^d))}
& \leq C \big( \| \mu \|_{L^{2}(0,T;V)} + \|\partial_t  \mu \|_{L^{2}(0,T;V^{*})} \big) \\
& \leq C \| f \|_{L^2(0,T;V^*)}
 \leq C \| \nabla \cdot (w m_2) \|_{L^{2}(0,T;V^{*})} \\
&  \leq C \| w m_2 \|_{L^{2}(Q;\mathbb{R}^d)}.
\end{align*}
Finally, since $\| m_2 \|_{L^\infty(Q)} \leq R$, we have
$\| w m_2 \|_{L^2(Q;\R^d)}^2
\leq C \int_Q |w|^2 m_2 \, \dd x \, \dd t.
$
Combining the two last obtained inequalities, we obtain the announced result.
\end{proof}

\subsection{HJB equation} \label{subsec:44}

\begin{lemma} \label{lemma:reg_H}
The Hamiltonian $H$ is differentiable with respect to $p$ and $H_p$ is differentiable with respect to $x$ and $p$. Moreover, $H$, $H_p$, $H_{px}$, and $H_{pp}$ are locally H{\"o}lder-continuous.
\end{lemma}

\begin{proof}
See \cite[Lemma 1]{BHP-schauder}.
\end{proof}

The analysis of the HJB equation relies on its connection with the value function of an optimal control problem, that was introduced in \eqref{mapping:u}.
This connection allows first to show a uniform bound for $\umap[\gamma,P]$.

\begin{lemma} \label{lem:u-L-infty}
Let $R> 0$ and let $(\gamma,P) \in \ballR$. There exists a constant $C(R) > 0$ such that $\| \umap[\gamma,P] \|_{L^\infty(Q)} \leq C(R)$ and such that $u$ is $C(R)$-Lipschitz continuous with respect to $x$.
Moreover, for any $(x,t) \in Q$,
\begin{equation} \label{eq:u-bounded-controls}
\umap[\gamma,P](x,t) = \inf_{\nu \in \mathbb{L}_{\mathbb{F}}^{2,C(R)}(t,T)} J[\gamma,P](x,t,\nu).
\end{equation}
In the above relation, $\mathbb{L}_{\mathbb{F}}^{2,C(R)}(t,T)$ denotes the set of stochastic processes $\nu \in \mathbb{L}_{\mathbb{F}}^{2}(t,T)$ such that $\mathbb{E} \big[ \int_t^T |\nu_s|^2 \dd s \big] \leq C(R)$.
\end{lemma}

\begin{proof}
We first derive a lower bound of $L$.
By assumption \hypD{},
$L(x,t,0)$ and $L_v(x,t,0)$ are bounded. It follows then from the
strong convexity assumption (Assumption \hypA{})
that there exists a constant $C> 0$ such that
\begin{equation} \label{L_quad_growth1}
  \frac{1}{C} |\nu|^2 - C \leq L(x,t,\nu), \quad
  \text{ for all }(x,t,\nu) \in Q \times \R^d.
\end{equation}
Then, for any $(x,s) \in Q$ and for any $\nu \in \mathbb{R}^d$, we have the following estimates:
\begin{align} \nonumber
L(x,s,\nu) + \langle A^\star [P](x,s) , \nu \rangle & \geq  \frac{1}{C} |\nu|^2 - \|a\|_{L^\infty(Q;\mathbb{R}^{k \times d})} |P(s)| |\nu| - C \\
& \geq \frac{1}{C} (|\nu|^2 - |P(s)|^2 - 1) \geq \frac{1}{C} (|\nu|^2 - 1). \nonumber
\end{align}
Now we show that $\umap[\gamma,P]$ is bounded in $L^{\infty}(Q)$. For any $(x,t) \in Q$, using the above bound for the running cost $L$, the bound $\|\gamma\|_{L^{\infty}(Q)} \leq R$, together with Assumption \hypD{} on the terminal cost $g$, we obtain that
$\umap[\gamma,P](x,t) \geq - C(R)$.
In addition, using Assumption \hypC{} and the fact that that $\|\gamma\|_{L^{\infty}(Q)} \leq R$, we deduce that
\begin{equation*}
    \umap[\gamma,P](x,t) \leq J[\gamma,P](x,t,0) \leq C(R),
\end{equation*}
from which we conclude that $\| \umap[\gamma,P] \|_{L^\infty(Q)} \leq C(R)$.

Finally we show equation \eqref{eq:u-bounded-controls}.
Let $t\in [0,T]$, let $\varepsilon \in (0,1)$ and let $\tilde{\nu} \in \mathbb{L}_{\mathbb{F}}^2(t,T)$ be an $\varepsilon$-optimal process.
Since $g$ is bounded (Assumption \hypD{}) and since $(\gamma,P) \in \ballR$, we deduce from the above inequality that
\begin{align}  \nonumber
\mathbb{E} \Big[ \int_t^T |\tilde{\nu}_s|^2 \dd s \Big]  &
\leq C \Big( \, \inf_{\nu \in \mathbb{L}_{\mathbb{F}}^2(t,T)} J[\gamma,P](x,t,\nu) +  \varepsilon + 1 \Big) \\
&  \leq C \big(\umap[\gamma,P](x,t) +  2 \big) \leq C, \nonumber
\end{align}
where the constant $C$ does not depend on $t$ and $\varepsilon$.
Thus any $\varepsilon$-optimal process lies in $\mathbb{L}_{\mathbb{F}}^{2,C}(t,T)$, which concludes the proof.
\end{proof}

\begin{proof}[Proof of Lemma \ref{lemma:hjb2}]
\label{proof:hjb2}
Let $(\gamma_1,P_1)$ and $(\gamma_2,P_2)$ be in $\ballR$. Let $u_1= \umap[\gamma_1,P_1]$ and $u_2= \umap[\gamma_2,P_2]$.
By Lemma \ref{lem:u-L-infty}, there exists $C>0$ such that
\begin{equation} \nonumber
u_2(x,t)-u_1(x,t)  = 
\inf_{\nu \in \mathbb{L}_{\mathbb{F}}^{2,C}(t,T)}
J[\gamma_2,P_2](x,t,\nu)
-
\inf_{\nu' \in \mathbb{L}_{\mathbb{F}}^{2,C}(t,T)}
J[\gamma_1,P_1](x,t,\nu'),
\end{equation}
for any $(x,t) \in Q$.
We denote $(X^\nu_{s})_{s \in [t,T]}$ the solution to the stochastic differential equation $\dd X_s = \nu_s \dd s + \sqrt{2} \dd B_s$ with $X^\nu_t = x$, for any $\nu \in \mathbb{L}_{\mathbb{F}}^{2}(t,T)$. Then
\begin{align}  \nonumber
& | u_2(x,t)- u_1(x,t)| 
\leq \sup_{\nu \in \mathbb{L}_{\mathbb{F}}^{2,C}(t,T)} \big| J[\gamma_2,P_2](x,t,\nu) - J[\gamma_1,P_1](x,t,\nu) \big| \\
& \qquad \quad \leq \sup_{\nu \in \mathbb{L}_{\mathbb{F}}^{2,C}(t,T)} \mathbb{E} \Big[ \int_t^T  |\langle A^\star [P_2 - P_1](X^\nu_s,s) , \nu_s \rangle| + |(\gamma_2- \gamma_1)(X^\nu_s,s)| \, \dd s \Big].  \nonumber
\end{align}
For any $(x,s) \in Q$ and $\nu \in \mathbb{R}^d$, the Cauchy-Schwarz inequality  yields
\begin{align} \nonumber
|\langle A^\star [P_2 - P_1](x,s) , \nu \rangle | & \leq |\langle a(x,s) P_2(s)- P_1(s) | |\nu | \\ &
\leq \|a\|_{L^\infty(Q;\mathbb{R}^{k \times d})} | P_2(s) - P_1(s)|
\, |\nu |.
\nonumber
\end{align}
Using again Cauchy-Schwarz inequality and $\| a\|_{L^\infty(Q;\mathbb{R}^{k \times d})} \leq C$, we conclude that
\begin{equation}  \nonumber
| u_2(x,t)- u_1(x,t)| \leq C \Big(\| P_2- P_1 \|_{L^2(0,T;\mathbb{R}^k)}  + \| \gamma_2- \gamma_1 \|_{L^\infty(Q)} \Big),
\end{equation}
as was to be proved.
\end{proof}

We prove Proposition \ref{prop:hjb1} with a density argument. In a nutshell: we prove in Proposition \ref{prop:hjb_hol} below that the result of Proposition \ref{prop:hjb1} holds true when $\gamma$ and $P$ are H\"older continuous. Then we pass to the limit, using Lemma \ref{lemma:hjb2}.

\begin{proposition} \label{prop:hjb_hol}
Let $R>0$ and let $\beta \in (0,1)$. For any $(\gamma,P) \in \ballR \cap \mathcal{C}^\beta(Q) \times \mathcal{C}^\beta(0,T;\R^k)$, the viscosity solution to the Hamilton-Jacobi-Bellman equation \eqref{eq:hjb_alone} is a classical solution. Moreover, there exists $\alpha \in (0,1)$ such that $\umap[\gamma,P]$ lies in $\mathcal{C}^{2+\alpha,1+\alpha/2}(Q)$ and there exists a constant $C(R)$, depending only on $R$, such that $\| \umap[\gamma,P] \|_{W^{2,1,q}(Q)} \leq C$.
\end{proposition}

The proof of Proposition \ref{prop:hjb_hol} is given at page \pageref{proof:hjb_hol} and relies on a fixed point approach which requires some preparatory work. We introduce the map $\mathcal{T} \colon W^{2,1,q}(Q) \times [0,1] \to W^{2,1,q}(Q)$ which associates to any $u \in W^{2,1,q}(Q)$ and $\tau \in [0,1]$ the classical solution $\tilde{u} = \mathcal{T}[u,\tau]$ to the linear parabolic equation
\begin{equation} \nonumber
\begin{array}{rlr}
- \partial_t \tilde{u} - \Delta \tilde{u} + \tau \bm{H}[\nabla u + A^\star P] \, = & \!\!\! \tau \gamma & (x,t) \in Q, \\
 \tilde{u}(x,T) \, = & \!\!\!  \tau g(x) & x\in \mathbb{T}^d.
\end{array}
\end{equation}
For any $(u,\tau) \in W^{2,1,q}(Q) \times [0,1]$, we have $\tau ( \gamma - \bm{H}[\nabla u + A^\star P]) \in L^\infty(Q)$, by Lemma \ref{lemma:reg_H} and Lemma \ref{lemma:max_reg_embedding}.
It follows then from Theorem \ref{theo:max_reg1} that $\mathcal{T}[u,\tau]$ lies in $W^{2,1,q}(Q)$, proving that $\mathcal{T}$ is well-defined.

\begin{lemma} \label{lemma:cont-comp-T}
 The mapping $\mathcal{T}$ is continuous and compact. In addition, for all $K >0$, there exists $\alpha \in (0,1)$ and $C>0$ depending on $K$, $\gamma$, and $P$ such that $\|u\|_{W^{2,1,q}(Q)} \leq K$ implies $\|\mathcal{T}[u,\tau]\|_{\mathcal{C}^{2+\alpha,1+ \alpha/2}(Q)} \leq C$.
\end{lemma}

\begin{proof}
\noindent \textit{Step 1: Continuity of $\mathcal{T}$.}
Let $(u_k,\tau_k) \in W^{2,1,q}(Q) \times [0,1]$ be a sequence converging to $(u,\tau) \in W^{2,1,q}(Q) \times [0,1]$. Then $\nabla u_k \to \nabla u$ in $L^\infty(Q;\mathbb{R}^d)$ by Lemma \ref{lemma:max_reg_embedding}. Then $\tau_k (\gamma - \bm{H}[\nabla u_k + A^\star P]) \to \tau (\gamma - \bm{H}[\nabla u + A^\star P])$ in $L^\infty(Q;\mathbb{R}^d)$ by continuity of the Hamiltonian (see Lemma \ref{lemma:reg_H}). Finally  $\mathcal{T}$ is continuous, by Theorem \ref{thm:bound-u-u0-h}.

\noindent \textit{Step 2: Compactness of $\mathcal{T}$.} Let $K>0$ and let $(u,\tau) \in W^{2,1,q}(Q) \times [0,1]$ be such that $\|u\|_{W^{2,1,q}(Q)} \leq K$. 
Combining Lemma \ref{lemma:max_reg_embedding} and Lemma \ref{lemma:reg_H} there exist $\alpha \in (0,1)$ and $C>0$ such that $ \|\gamma - \bm{H}[\nabla u + A^\star P]\|_{C^\alpha(Q)} \leq C$. Then applying Theorem \ref{theo:holder_reg_classical}, there exist $\alpha \in (0,1)$ and $C>0$ such that $\|\mathcal{T}[u,\tau]\|_{\mathcal{C}^{2+\alpha,1+ \alpha/2}(Q)} \leq C$. By the Arzela-Ascoli Theorem the centered ball of $\mathcal{C}^{2+\alpha,1+ \alpha/2}(Q)$ of radius $C>0$ is a relatively compact subset of $W^{2,1,q}(Q)$. As a consequence $\mathcal{T}[u,\tau]$ is a compact mapping and the conclusion follows.
\end{proof}

\begin{theorem}{\normalfont (Leray-Schauder)} \label{thm:Leray-schauder}
Let $X$ be a Banach space and let $T : X \times [0, 1] \to X$ be a continuous and compact mapping. Assume that $T (x,0) = 0$ for all $x\in X$ and assume there exists $C>0$ such that $\|x\|_X < C$ for all $(x,\tau) \in X\times [0,1]$ such that $T (x,\tau) = x$. Then, there exists $x \in X$ such that $T(x,1) = x$.
\end{theorem}

\begin{proof}
See \cite[Theorem 11.6]{gilbarg2015elliptic}.
\end{proof}

\begin{proof}[Proof of Proposition \ref{prop:hjb_hol}] \label{proof:hjb_hol}
We prove that under the assumptions of the proposition, the HJB equation has a classical solution in $\mathcal{C}^{2+\alpha,1+\alpha/2}(Q)$ (for some $\alpha \in (0,1)$), which is then necessarily the unique viscosity solution $\umap[\gamma,P]$. To this purpose, we prove the existence of a solution to the fixed point equation 
$u= \mathcal{T}[u,1]$.
We have $\mathcal{T}[u,0] = 0$ for all $u \in W^{2,1,q}(Q)$. Now let $(u,\tau) \in W^{2,1,q}(Q) \times [0,1]$ be such that $\mathcal{T}[u,\tau] = u$. From Lemma \ref{lemma:cont-comp-T}, the mapping $\mathcal{T}$ is continuous and compact, in addition $u$ is a classical solution and thus the viscosity solution to the Hamilton-Jacobi-Bellman equation
\begin{equation} \nonumber
\begin{array}{rlr}
- \partial_t u - \Delta u + \tau \bm{H}[\nabla u + A^\star P] \, = & \!\!\! \tau \gamma & (x,t) \in Q, \\
 u(x,T) \, = & \!\!\! \tau g(x) & x\in \mathbb{T}^d,
\end{array}
\end{equation}
and can be interpreted as the value function associated to the following stochastic control problem
\begin{equation} \nonumber
\inf_{\nu \in \mathbb{L}_{\mathbb{F}}^2(0,T)} 
 \tau \mathbb{E} \Big[\int_0^T  L(X^\tau_s,s,\nu_s) + \langle A^\star [P](X^\tau_s,s) , \nu_s \rangle + \gamma(X^\tau_s,s) \dd s + g(X^\tau_T) \Big],
\end{equation}
where $(X^\tau_{s})_{s \in [t,T]}$ is the solution to $\dd X_s = \tau \nu_s \dd s + \sqrt{2} \dd B_s$, $X_0 = Y$. Following \cite[Proposition 1, Step 2]{BHP-schauder}, there exists a constant $C>0$, depending only on $R$, such that $\|u\|_{L^\infty(Q)} + \|\nabla u\|_{L^\infty(Q;\mathbb{R}^d)} \leq C$.
Then using Lemma \ref{lemma:reg_H} and recalling that $(\gamma,P) \in \Xi_R$, we deduce that $\| \bm{H}[\nabla u + A^\star P] - \gamma\|_{L^\infty(Q)} \leq C$. It follows that $u$ is the solution to a parabolic PDE with bounded coefficients and thus $\| u \|_{W^{2,1,q}(Q)} \leq C$, by Theorem \ref{theo:max_reg1}. Again, $C$ only depends on $R$.
Finally, by the Leray-Schauder theorem (Theorem \ref{thm:Leray-schauder}), there exists a solution to $u= \mathcal{T}[u,1]$, which is necessarily $\umap[\gamma,P]$.
\end{proof}

\begin{proof}[Proof of Proposition \ref{prop:hjb1}] \label{proof:hjb1}
Take $(\gamma,P) \in \ballR$ and fix $\beta \in (0,1)$.
Let $(\gamma_n,P_n)$ be a sequence in ${\Xi}_{R+1} \cap \mathcal{C}^{\beta}(Q) \times \mathcal{C}^{\beta}(0,T;\R^k)$ such that
$\| \gamma_n - \gamma \|_{L^\infty(Q)} \longrightarrow 0$ and such that $\| P_n - P \|_{L^2(0,T;\R^k)} \longrightarrow 0$.
We do not detail the construction of such a sequence, this can be done by convolution.
Define $u^n= \umap[\gamma^n,P^n]$ and $u=\umap[\gamma,P]$. By Lemma \ref{lemma:hjb2}, $u_n \rightarrow u$ for the $L^\infty$-norm.
Moreover, by Proposition \ref{prop:hjb_hol},
\begin{equation} \label{eq:boundW21p}
\| u^n \|_{W^{2,1,q}(Q)} \leq C(R), \quad \forall n \in \mathbb{N}.
\end{equation}
Thus, the three sequences $(\partial_t u^n)_{n \in \mathbb{N}}$, $(\Delta u^n)_{n \in \mathbb{N}}$, and $(\nabla u^n)_{n \in \mathbb{N}}$ are bounded in $L^q(Q)$. By the Banach-Alaoglu theorem, the three sequences have at least one accumulation point for the weak topology of $L^q(Q)$. These three accumulation points are necessarily (by definition of weak derivatives) equal to $\partial_t u$, $\Delta u$, and $\nabla u$, respectively. Since
the $L^q$-norm is weakly lower semi-continuous, we deduce that $\| u \|_{W^{2,1,q}(Q)} \leq C(R)$, where $C(R)$ is as in \eqref{eq:boundW21p}. This concludes the proof.
\end{proof}

\subsection{The other mappings}

\begin{proof}[Proof of Lemma \ref{lemma:other}] \label{proof:other}

Let $(\gamma,P) \in \ballR$. Let $u= \umap[\gamma,P]$. We already know  from Proposition \ref{prop:hjb1} that $\| u \|_{W^{2,1,q}(Q)} \leq C(R)$. Then Lemma \ref{lemma:max_reg_embedding} implies that $u$ and $\nabla u$ are continuous and that
$\| u \|_{L^\infty(Q)} \leq C(R)$ and $\| \nabla u \|_{L^\infty(Q;\R^d)} \leq C(R)$.
Let $v=\vmap[\gamma,P]= -  \bm{H}_p[  \nabla u +  A^\star P]$.
We have
\begin{equation*}
D_x v = - \bm{H}_{px}[\nabla u + A^\star P] - \bm{H}_{pp}[\nabla u + A^\star P](D^2_{xx} u + D_x A^\star P).
\end{equation*}
Using the regularity of $u$, the regularity properties of the Hamiltonian given in Lemma \ref{lemma:H-L-quad}, and the regularity assumptions on $a$ (Assumption \hypD{}), we deduce that $\| v \|_{L^\infty(Q;\R^d)} \leq C(R)$ and that $\|D_x v \|_{L^q(Q;\R^{d \times d})} \leq C(R)$. Moreover, $v$ is continuous.

Next, let $m= \mmap[\gamma,P]= \fpmap[v]$. A direct application of Lemma \ref{lemma:fp} yields that $\| m \|_{W^{2,1,q}(Q)} \leq C$.
Finally, let $w= \wmap[\gamma,P]= mv$.
Using again Lemma \ref{lemma:max_reg_embedding}, we obtain that $m$ is continuous and that $\| m \|_{L^\infty(Q)} \leq C(R)$ and $\| \nabla m \|_{L^\infty(Q;\R^d)} \leq C(R)$. Then $w \in \feedback{}$, with a norm bounded by some constant $C(R)$. The lemma is proved.
\end{proof}

\begin{proof}[Proof of Lemma \ref{lemma:coupling}]
\label{proof:coupling}

The two statements concerning $\gammamap$ are directly deduced from Assumptions \hypC{} and \hypE{}.
Let $w \in \feedback{}$. Recalling the definition of the operator $A$ (page \pageref{def:A}), it is easy to see with Assumption \hypD{} that $Aw \in \mathcal{C}(0,T;\R^d)$. Assumptions \hypC{} and \hypE{} ensure then that $\Pmap[w]= \bm{\phi} \big[ A[w] \big]$ lies in $\mathcal{C}(0,T;\R^k)$ and that $\| \Pmap[w] \|_{L^\infty(0,T;\R^k)} \leq C$.
Let us next consider $w_1$ and $w_2$ in $\feedback{}$. We have
\begin{align*}
\| Aw_2 - Aw_1 \|_{L^\infty(0,T;\R^k)}
\leq {} & \| a \|_{L^\infty(Q;\R^{k \times d})}
\| w_2- w_1 \|_{L^\infty(0,T;L^1(\mathbb{T}^d;\R^d))} \\
\leq {} & C \| w_2 - w_1 \|_{L^2(Q;\R^d)},
\end{align*}
by Assumption \hypD{}. Using next the Lipschitz-continuity of $\phi$ (Assumption \hypE{}), we obtain that
$\| \bm{\phi}[Aw_2] - \bm{\phi}[Aw_1] \|_{L^2(0,T;\R^k)}
\leq C \| w_2-w_1 \|_{L^2(Q;\R^d)}$,
as was to be proved.
\end{proof}

\end{document}